\documentclass[10pt]{article}
\usepackage[square,authoryear]{natbib}
\usepackage{graphicx}
\usepackage[all]{xy}
\usepackage{marsden_article}

\usepackage{setspace}
\newcommand{\lb}{\ensuremath{\langle}}
\newcommand{\rb}{\ensuremath{\rangle}}

\newcommand{\extd}{\ensuremath{\mathbf{d}}}

\DeclareMathOperator{\pr}{pr}

\title{Tensor Products of Dirac Structures and Interconnection in Lagrangian Mechanics}

\author{
\hspace{-1cm} \begin{tabular}{cc}
Henry Jacobs &
Hiroaki Yoshimura\thanks{Research partially supported by JSPS (23560269), JST-
CREST, Waseda University (2012A-602), and the IRSES
project ``Geomech" (246981) within the 7th European Com-
munity Framework Programme.}
\\  Department of Mathematics & Applied Mechanics and Aerospace Engineering
\\ Imperial College London & Waseda University
\\  London SW7 2AZ, United Kingdom & Okubo, Shinjuku, Tokyo 169-8555, Japan \\ h.jacobs@imperial.ac.uk & yoshimura@waseda.jp\\
\end{tabular}\\\\
}

\date{\today}

\begin{document}

\maketitle

\begin{center}
{\it Dedicated to the memory of Jerrold E. Marsden\\ }
\end{center}

\begin{abstract}
Many mechanical systems are large and complex, despite being composed of simple subsystems.
In order to understand such large systems it is natural to tear the system into these subsystems.
Conversely we must understand how to invert this tearing.
In other words, we must understand \emph{interconnection}.
Such an understanding has already successfully understood in the context of Hamiltonian systems on vector spaces via 
the port-Hamiltonian systems program.
In port-Hamiltonian systems theory, interconnection is achieved through the identification of shared variables, whereupon the notion of \emph{composition of Dirac structures}
allows one to interconnect two systems.
In this paper we seek to extend the port-Hamiltonian systems program to Lagrangian systems on manifolds and extend the notion of \emph{composition of Dirac structures} appropriately.
In particular, we will interconnect Lagrange-Dirac systems by modifying the respective Dirac structures of the involved subsystems.
We define the {\it interconnection of Dirac structures} via an {\it interaction Dirac structure} and a {\it tensor product} of Dirac structures.
We will show how the dynamics of the interconnected system is formulated as a function of the subsystems, and we will elucidate the associated variational principles.
We will then illustrate how this theory extends the theory of port-Hamiltonian systems and the notion of \emph{composition of Dirac structures} to manifolds with couplings which do not require the identification of shared variables.
Lastly, we will close with some examples: a mass-spring mechanical systems, an electric circuit, and a nonholonomic mechanical system.
\end{abstract}

\tableofcontents

\section{Introduction}
A large class of physical and engineering systems can be described as constrained or unconstrained Lagrangian or Hamiltonian systems.
However, the analysis of these systems is difficult when the dimensions get large and as the structures becomes more heterogeneous.
For example, consider systems which involve a mixture of mechanical and electrical components with flexible and rigid parts and magnetic couplings (e.g. \cite{Yo1995, Bloch2003,ABHM2006}).
To handle these complex situations, it is natural to \emph{tear} the system into simpler subsystems.
However, once one tears, one is left with a number of disconnected subsystems with undefined dynamics.
The final step in obtaining the dynamics of the connected system is what we call \emph{interconnection}.
In describing interconnected systems, the use of Dirac structures has become standard.
 
  Over the past few decades, Dirac structures have emerged as generalization of symplectic and Poisson structures providing a new perspective on the Hamiltonian formalism (see \cite{Cour1990}).
  Secondly, the Hamilton-Pontryagin variational principle has allowed the Lagrangian formalism (including degenerate Lagrangians) to be written in terms of Dirac structure (see \cite{YoMa2006b}).
  As a result, there is now a formalism which one could call the Dirac formalism which generalizes the Hamiltonian and Lagrangian formalisms.
  In this paper we will consider the interconnection of Lagrange-Dirac dynamical systems where the dynamics can be obtained through \emph{Dirac structures}.

	A Dirac structure is a type of power-conserving relation on a phase space, such as a kinematic constraint, Newton's third law, or a magnetic coupling.
	In particular, we call the Dirac structures which express power-conserving couplings ``interaction Dirac structures.''
	The question we seek to answer is ``how can we use interaction Dirac structures to perform interconnections?''
	More specifically, given mechanical systems with Dirac structures $D_1$ and $D_2$ on manifolds $M_1$ and $M_2$, how do we use an interaction Dirac structure, $D_{\mathrm{int}}$ on $M_1 \times M_2$ ?
	The key ingredient is the Dirac tensor product, denoted $\boxtimes$ (see \cite{Gua2011}) and the answer we propose is that the Dirac structure of the interconnected Lagrange-Dirac system is
	\[
		D_{C} := (D_1 \oplus D_2) \boxtimes D_{\mathrm{int}},
	\]
	which is a Dirac structure over $M_1 \times M_2$.
	We will find that $D_{C}$ is the Dirac structure for the system which couples the Dirac systems on Dirac manifolds $(M_1,D_1)$ and $(M_2,D_2)$ using the power-conserving coupling given by $D_{\mathrm{int}}$.

\section{Background}
An early example of interconnection may be traced back to Gabriel Kron in his book, ``Diakoptics''  (\cite{Kr1963}).
The word ``diakoptics'' denotes the procedure of {\it tearing} a dynamical system into well-understood subsystems.
Each tearing is associated with a constraint on the interface between the two subsystems
and the original system is restored by {\it interconnecting} the subsystems with these constraints.
Kron's theory was further developed to handle power conserving interconnections in the form of bond graph theory (see \cite{Pay1961}).
Additionally, this was later specialized to electrical networks through Kirchhoff's current and voltage laws and the notion of a (nonenergic) multiport (\cite{Bra1971,WyCh1977}).
In mechanics, kinematic constraints due to mechanical joints, nonholonomic constraints, and force equilibrium conditions in d'Alembert's principle lead to these interconnections (\cite{Yo1995}).
In this paper we explore how a Dirac structure can play the role of a nonenergic multiport.

\paragraph{Dirac Structures in Mechanics.}
 In physical and engineering problems, Dirac structures can provide a natural geometric framework for describing interconnections between ``easy-to-analyze'' subsystems.
This is especially evident in the vast and growing literature of port-Hamiltonian systems (see for instance \cite{Van96} and references therein).
As mentioned, Dirac structures generalize Poisson and pre-symplectic structures and hence one can deal with implicitly defined equations of motion for mechanical systems with nonholohomic constraints.
This transition away from Poisson structures and Hamilton's principle induces a transition from ODEs to DAEs, in which case we call the resulting Hamiltonian or Lagrangian systems \emph{Hamilton-Dirac or Lagrange-Dirac systems}.
  In particular, \cite{VaMa1995} demonstrated how certain interconnections could be described by Dirac structures associated to constrained Poisson structures and provided an example of an L-C circuit as a Hamilton-Dirac ({\it implicit Hamiltonian}) system.
  On the Lagrangian side,  \cite{YoMa2006a} showed that nonholonomic mechanical systems and L-C circuits (as degenerate Lagrangian systems) could be formulated as Lagrange-Dirac ({\it implicit Lagrangian}) systems associated with Dirac structures induced from relevant constraint distributions.
  Finally, \cite{YoMa2006b} demonstrated how the {\it implicit Euler-Lagrange equations} for unconstrained systems could be derived from the {\it Hamilton-Pontryagin principle} and how constrained Lagrange-Dirac systems with forces could be formulated in the context of the {\it Lagrange-d'Alembert-Pontryagin principle}. 

\paragraph{Port-Controlled Hamiltonian and Lagrangian Systems.}
  In the realm of control theory, {\it implicit port-controlled Hamiltonian (IPCH) systems} (systems with external control inputs) were developed by \cite{VaMa1995} (see also \cite{BC1997}, \cite{BL2000} and \cite{Van96}) and much effort has been devoted to understanding passivity based control for interconnected IPCH systems (\cite{Ort1998}).
  This perspective builds upon bond-graph theory and has proven useful in deriving equations of motion especially in the context on multi-components systems.  For instance, \cite{Duindam_thesis} used port-based methodologies to describe a controller for a robotic walker.  An overview on the application of port-Hamiltonian systems to controller design for electro-mechanical is given in chapter 3 of (\cite{modelling2009}).

 With regards to theory, the equivalence between {\it controlled Lagrangian (CL) systems} and {\it controlled Hamiltonian (CH)} systems was shown by \cite{ChBlLeMa2002} for non-degenerate Lagrangians.  For the case in which the Lagrangian is degenerate, an implicit Lagrangian analogue of IPCH systems, namely, an {\it implicit port-controlled Lagrangian (IPCL) systems} for electrical circuits were constructed by \cite{YoMa2006c} and \cite{YoMa2007a}, where it was shown that L-C transmission lines can be represented in the context of the IPLC system by employing induced Dirac structures.

 The notion of  \emph{composition of Dirac structures} was developed in \cite{CeVdeScBa2007} for the purpose of interconnection in IPCH systems.
  This provided a new tool for the passive control of IPCH systems.
  In particular, it was shown that the feedback interconnection of a ``plant'' port-Hamiltonian system with a ``controller'' port-Hamiltonian system could be represented by the composition of the plant Dirac structure with the controller Dirac structure.
  While the construction was originally restricted to the case of linear Dirac structures on vector spaces, these constructions have been generalized to the case of manifolds where the ports are modeled with trivial vector-bundles and with flat Ehresmann connections by \cite{Merker2009}.
  However, the existence of a flat Ehresmann connection is not guaranteed on arbitrary vector bundles.
  Therefore, in order to apply the notion of interconnections to Lagrangian systems, we will extend the notion of composition of Dirac structures to the general case of interconnection by constraint distributions on manifolds.
  This extension is the main contribution of the paper.

\section{Main Contributions.}
The main purpose of this paper is to elucidate Kron's notion of interconnections in the context of induced Dirac structures.
To do this, we consider two sub-systems whose equations of motion are given by Lagrange-Dirac systems and with Dirac structures $D_1$ and  $D_2$ on manifolds $M_1$ and $M_2$ respectively.
Then we show how an interconnection between these sub-systems is represented by a Dirac structure, $D_{\mathrm{int}}$ on $M_1 \times M_2$.
We will observe that the connected system is a Lagrange-Dirac system, whose Lagrangian is the sum of the Lagrangians of the sub-systems and whose Dirac structure is given by the formula $D_C = (D_1 \oplus D_2) \boxtimes D_{\rm int}$, where $\boxtimes$ is the Dirac tensor product introduced in \cite{Gua2011}.
More generally, the interconnection of $N$ systems can be done with a single interconnection Dirac structure, $D_{\rm int}$, and the Dirac structure of the interconnected system is given by
 \[
   	\underbrace{ D_C }_{ \text{interconnected} } = \overbrace{ ( D_1 \oplus \cdots \oplus  D_n )}^{ \text{sub-systems} }  \!\!\!\! \!\!\!\! \underbrace{ \boxtimes }_{ \text{ tensor product } } \!\!\!\!\!\!\!\! \overbrace{ D_{\text{int}}}^{\text{ interaction }}\!\!\!\!\!.
 \] 

We do this through the following sequence: In \S \ref{sec:Review}, we briefly review Dirac structures in Lagrangian mechanics following \cite{YoMa2006a,YoMa2006b}.
In \S \ref{sec:tensor_products}, we show how a power-conserving interconnection can be represented by a Dirac structure (usually labeled $D_{\mathrm{int}}$ in this paper) and how one could obtain the Dirac structure of the interconnected system using the {\it tensor product}, $\boxtimes$.
In particular, we have been influenced by the notion of composition of Dirac structures introduced for the purpose of interconnection in \cite{CeVdeScBa2007}.
The constructions we will present modify this notion so that it may be extended to the case of manifolds using intrinsic expressions and a fairly general class of power-conserving couplings which are representable by Dirac structures.
Moreover, we provide an explicit translation of (\cite{CeVdeScBa2007}) into the constructions presented here.
In \S \ref{sec:interconnection}, we explore how this procedure alters the variational structure of Lagrange-Dirac dynamical systems.
In \S \ref{sec:examples}, we demonstrate our theory by applying it to an LCR circuit, a nonholonomic system, and a simple mass-spring system.

\paragraph{Acknowledgements.}    
We are very grateful to Henrique Bursztyn, Mathieu Desbrun, Fran\c{c}ois Gay-Balmaz, Humberto Gonzales, Melvin Leok, Tomoki Ohsawa, Arjan van der Schaft, Joris Vankerschaver, Tudor Ratiu, Jedrzej Sniatycki, and Alan Weinstein for useful remarks and suggestions. The research of H.~J. was funded by NSF grant CCF-1011944. The research of H.~Y. is partially supported by JSPS Grant-in-Aid 23560269, JST-CREST and Waseda University Grant for SR 2012A-602.

\paragraph{Notation and Conventions}
 In this paper, all objects are assumed to be smooth.  Given a manifold $M$, we denote the tangent bundle by $\tau_M: TM \to M$ and the cotangent bundle by $\pi_M: T^{\ast}M \to M$.  Given a fiber bundle, $\pi: F \to M$ we denote the set of sections of $F$ by $\Gamma(F)$.  Lastly, given a second manifold $N$ and a map $f: M \to N$ we denote the tangent lift by $Tf:TM \to TN$, and if $f$ is a diffeomorphism, we may denote the cotangent lift by $T^{\ast}f:T^{\ast}N \to T^{\ast}M$.

\section{Review of Dirac Structures in Lagrangian Mechanics}
\label{sec:Review}

\paragraph{Linear Dirac Structures.} As in \cite{CoWe1988}, we start with finite dimensional vector spaces before going to manifolds.  Let $V$ be a finite dimensional vector space and let $V^{\ast}$ be the dual space, where we denote the natural pairing between $V^{*}$ and $V$ by $\langle\cdot \, , \cdot\rangle$. Define the symmetric pairing
$\langle \! \langle\cdot,\cdot \rangle \!  \rangle$
on $V \oplus V^{\ast}$ by
\begin{equation*}
\langle \! \langle\, (v,\alpha),
(\bar{v},\bar{\alpha}) \,\rangle \!  \rangle
=\langle \alpha, \bar{v} \rangle+\langle \bar{\alpha}, v \rangle,
\end{equation*}
for any $(v,\alpha), (\bar{v},\bar{\alpha}) \in V \oplus V^{\ast}$.
\medskip

A {\it constant Dirac structure} on $V$ is a maximally isotropic subspace $D \subset V \oplus V^{\ast}$ such that $D=D^{\perp}$, where $D^{\perp}$ is the orthogonal complement of $D$ relative to $\langle \! \langle \cdot,\cdot \rangle \!  \rangle$.

\paragraph{Dirac Structures on Manifolds.} 
Let $M$ be a smooth manifold and we denote by $TM \oplus T^{\ast}M$ the Pontryagin bundle, which is the Whitney sum bundle over $M$, namely, the bundle over the base $M$ and with fiber over $x \in M $ equal to $T_xM \times T_x^{\ast}M$. A subbundle, $ D \subset TM \oplus T^{\ast}M$, is called an \emph{almost Dirac structure} on $M$, when $D(x)$ is a Dirac structure on the vector space $T_{x}M$ at each $x \in M$. We can define an almost Dirac structure from a two-form $\Omega$ on $M$ and a regular distribution $\Delta_{M}$ on $M$ as follows: For each $x \in M$, set
\begin{equation}\label{DiracManifold}
\begin{split}
D(x)=\{ (v, \alpha) \in T_xM \times T_x^{\ast}M
  \; \mid \; & v \in \Delta_{M}(x), \; \mbox{and} \\ 
  & \lb \alpha , w \rb =\Omega_{\Delta_{M}}(x)(v,w) \; \;
\mbox{for all} \; \; w \in \Delta_{M}(x) \},
\end{split}
\end{equation}
where $\Delta_{M}^{\circ}$ is the annihilator of $\Delta_{M}$.

\paragraph{Integrablity.}
We call  $D$ an \emph{integrable Dirac structure} if the integrability condition
\begin{equation}\label{ClosedCond}
\langle \pounds_{X_1} \alpha_2, X_3 \rangle
+\langle \pounds_{X_2} \alpha_3, X_1 \rangle+\langle \pounds_{X_3}
\alpha_1, X_2 \rangle=0
\end{equation}
is satisfied for all pairs of vector fields and one-forms $(X_1, \alpha_1)$,
$(X_2,\alpha_2)$, $(X_3,\alpha_3)$ that take values in $D$,
where $\pounds_{X}$ denotes the Lie derivative  along the vector
field $X$ on $M$.

\paragraph{Remark.} Let $\Gamma(TM \oplus T^{\ast}M)$ be a space of local sections of $TM \oplus T^{\ast}M$, which is endowed with the skew-symmetric bracket
$\left[ \cdot ,\cdot \right]: \Gamma(TM \oplus T^{\ast}M) \times  \Gamma(TM \oplus T^{\ast}M) \to  \Gamma(TM \oplus T^{\ast}M)$ defined by
$$
\left[(X,\alpha),(Y,\beta)\right] := \left( \left[X,Y \right],  \pounds_{X} \beta- \pounds_{Y} \alpha + \frac{1}{2}\mathbf{d} (\alpha(Y)-\beta(X)) \right).
$$
This bracket was originally given in \cite{Cour1990} and does not necessarily satisfy the Jacobi identity. It was shown by \cite{Dorfman1993} that the integrability condition of the Dirac structure $D \subset TM \oplus T^{\ast}M$ given in equation \eqref{ClosedCond} can be expressed as
\[
\left[\Gamma(D), \Gamma(D) \right] \subset \Gamma(D), 
\]
which is the closure condition with respect to the Courant bracket.  In particular, this closure condition is the Dirac structure analog of the closure condition of a symplectic structure or the Jacobi identity in the context of Poisson structures.

\paragraph{Induced Dirac Structures.} 
One of the most relevant Dirac structures for Lagrangian mechanics is derived from linear velocity constraints.
Such constraints are given by a regular distribution $\Delta_{Q} \subset TQ$ on a configuration manifold $Q$.
We can naturally derive a Dirac structure over $TT^{\ast}Q$ from $\Delta_{Q}$ using the constructions described in \cite{YoMa2006a}.
\medskip

  Define the lifted distribution on $T^{\ast}Q$ by
\begin{equation*}
\Delta_{T^{\ast}Q}
=( T\pi_{Q})^{-1} \, (\Delta_{Q}) \subset TT^{\ast}Q,
\end{equation*}
where $\pi_{Q}:T^{\ast}Q \to Q$ is the cotangent bundle projection. Let $\Omega$ be the canonical two-form on $T^{\ast}Q$.  Define a Dirac structure $D_{\Delta_{Q}}$ on $T^{\ast}Q$, whose fiber is given for each $(q,p) \in T^{\ast}Q$ by
\begin{align*}\label{inducedDirac}
D_{\Delta_{Q}}(q,p)
& =\{ (v, \alpha) \in T_{(q,p)}(T^{\ast}Q) \times T_{(q,p)}^{\ast}(T^{\ast}Q) \mid v \in
\Delta_{T^{\ast}Q}(q,p),  \; \mbox{and} \;  \nonumber
\\ & \hspace{3cm}
\lb \alpha , w \rb = \Omega_{\Delta_{Q}}(q,p)(v,w) \;\; \mbox{for
all} \;\; w \in \Delta_{T^{\ast}Q}(q,p)\}.
\end{align*}
This Dirac structure is called an {\it induced Dirac structure} and provides an instance of construction \eqref{DiracManifold}.

\paragraph{Local Expressions.} 
Let $V$ be a model space for $Q$ and let $U$ be an open subset of $V$, which is a chart domain on $Q$.
Then, $TQ$ is locally represented by $U \times V$, while $T^{\ast}Q$ is locally represented by $U \times V^{\ast}$. 
Further, $TT^{\ast}Q$ is locally represented by $(U \times V^{\ast}) \times (V \times
V^{\ast})$, while $T^{\ast}T^{\ast}Q$ is locally represented by $(U \times V^{\ast})
\times (V^{\ast} \times V)$. 
\medskip

Using $\pi_Q: T^{\ast}Q\rightarrow Q$ locally denoted by $(q,p) \mapsto q$ and its tangent map $T\pi_Q : TT^{\ast}Q \to TQ;\;(q, p, \delta{q},\delta{p}) \mapsto (q, \delta{q})$, it follows that 
\[
\Delta_{T^{\ast}Q} = \left\{ (q,p, \delta{q}, \delta{p} ) \in TT^{\ast}Q
\mid q \in U, \;\delta{q} \in \Delta(q) \right\}
\]
and the annihilator of $\Delta_{T^{\ast}Q}$ is locally represented as
\[
\Delta^{\circ}_{T^{\ast}Q} = \left\{(q,p,\beta, w) \in T^{\ast}T^{\ast}Q
\mid q \in U,\;\, \beta \in \Delta^{\circ} (q), \;\,  w=0 \right\}.
\]
Since we have the local formula $\Omega ^\flat(q,p) \cdot (q,p,\delta{q}, \delta{p}) = (q,p, - \delta{p}, \delta{q})$,
the condition
$$(q,p,\gamma,u)-\Omega^{\flat}(q,p) \cdot (q,p,\delta{q}, \delta{p}) \in
\Delta^{\circ}_{T^{\ast}Q}$$ for $(q,p,\gamma,u) \in T^{\ast}T^{\ast}Q$ reads
$
\gamma +\delta{p} \in \Delta^{\circ}(q)$ and  $ u-\delta{q}=0.
$
Thus, the induced Dirac structure on $T^{\ast}Q$ is locally represented by
\begin{equation}\label{localdirac}
\begin{split}
D_{\Delta_{Q}}(q,p)  &=
\left\{
\left( (\delta{q}, \delta{p}), (\gamma, u) \right) \mid
\delta{q} \in \Delta(q), \; u= \delta{q}, \; \right.
\left. \gamma +\delta{p} \in \Delta^{\circ}(q)
\right\},
\end{split}
\end{equation}
where $\Delta^{\circ}(q) \subset T^{\ast}_{q}Q$ is the annihilator of $\Delta (q) \subset T_{q}Q$.

\paragraph{Iterated tangent and cotangent bundles.}
  Here we recall the geometry of the iterated tangent and cotangent bundles
$TT^{\ast}Q$,  $T^{\ast}T^{\ast}Q$ and $T^{\ast}TQ$, as well as the Pontryagin bundle $TQ \oplus T^{\ast}Q$. 
Understanding the interrelations between these spaces allows us to better understand the interrelation between Lagrangian systems and Hamiltonian systems, especially in the context of Dirac structures.
In particular, there are two diffeomorphisms between $T^{\ast}TQ$, $TT^{\ast}Q$ and $T^{\ast}T^{\ast}Q$ which were thoroughly investigated in \cite{Tu1977} in the context of the generalized Legendre transform.
\medskip

We first define a natural diffeomorphism
\[
\kappa_{Q}: TT^{\ast}Q \to T^{\ast}TQ; \quad (q, p, \delta q, \delta p) \mapsto (q, \delta q, \delta p, p),
\]
where $(q,p)$ are local coordinates
of $T^{\ast}Q$ and $(q,p,\delta{q},\delta{p})$ are the corresponding
coordinates of $TT^{\ast}Q$, while $(q, \delta q, \delta p, p)$ are
the local coordinates of $T^{\ast}TQ$ induced by $\kappa_{Q}$. 

Second, there exists a natural diffeomorphism $\Omega^{\flat}:TT^{\ast}Q \to T^{\ast}T^{\ast}Q$ associated to the canonical symplectic structure 
$\Omega$, which is locally denoted by
$
(q,p,\delta{q},\delta{p}) \mapsto (q,p,-\delta{p}, \delta{q}),
$
and hence we can define a diffeomorphism $\gamma_{Q}: T^{\ast}TQ \to T^{\ast}T^{\ast}Q$ by
$$
\gamma_{Q}:=\Omega^{\flat} \circ \kappa_{Q}^{-1}; \quad (q,\delta{q},\delta{p},p) \mapsto (q,p,-\delta{p}, \delta{q}).
$$
On the other hand, the Pontryagin bundle is equipped with three natural projections
  \[
  \begin{split}  
    	\pr_Q&: TQ \oplus T^{\ast}Q \to Q;\; (q,\delta{q},p) \mapsto q,  \\
	\pr_{TQ}&: TQ \oplus T^{\ast}Q  \to TQ;\; (q,\delta{q},p) \mapsto (q,\delta{q}), \\
	\pr_{T^{\ast}Q}&: TQ \oplus T^{\ast}Q \to T^{\ast}Q;\; (q,\delta{q},p) \mapsto (q,p).
\end{split}
\]

These interrelations are summarized (and defined) in the commutative diagram shown in Figure  \ref{BundPic}.

\begin{figure}[h]
\begin{center}
\includegraphics[scale=.59]{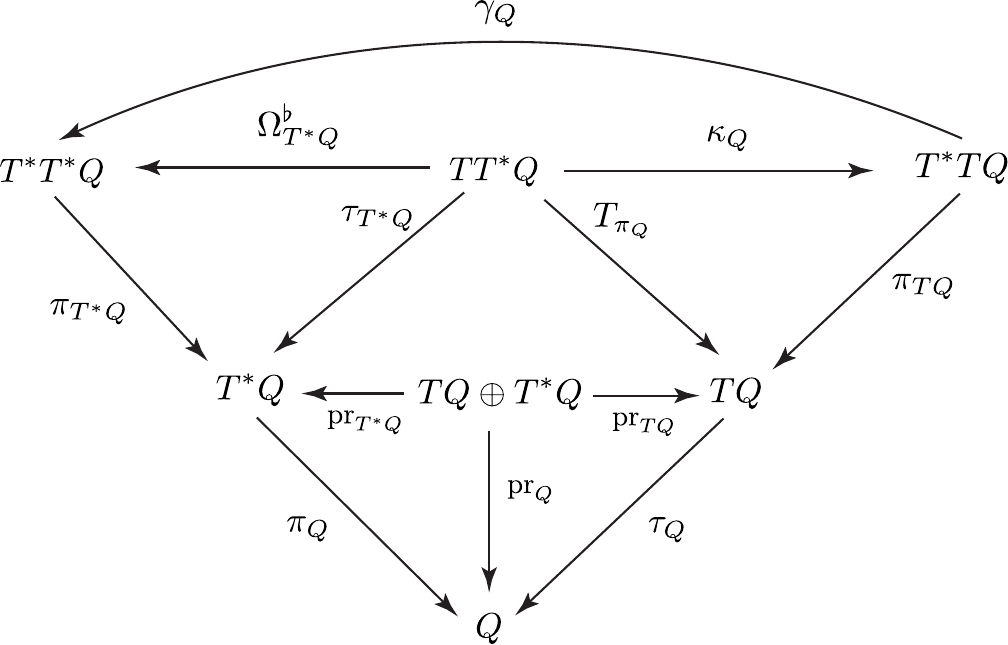}
\caption{The Bundle Picture}
\label{BundPic}
\end{center}
\end{figure}

\paragraph{Lagrange-Dirac Dynamical Systems.} 
Let $L:TQ \to \mathbb{R}$ be a Lagrangian, possibly degenerate. The differential $ \mathbf{d} L:TQ \rightarrow T^{\ast}TQ$ of $L$ is the one-form on $TQ$ which is locally given by, for each $(q,v) \in TQ$,
\[
\mathbf{d} L(q,v)= \left( q,v, \frac{\partial L}{\partial q}, \frac{\partial L}{\partial v}\right).  
\]

Using the canonical diffeomorphism $ \gamma_{Q}:T^{\ast}TQ \rightarrow T^{\ast}T^{\ast}Q$, we define the {\bfi Dirac differential} of $L$ by 
\[
\mathbf{d}_{D} L:= \gamma_{Q} \circ \mathbf{d} : TQ \to T^{\ast}T^{\ast}Q,
\]
which may be locally given by
$$
\mathbf{d}_{D} L(q,v)= \left(q,\frac{\partial L}{\partial v}, - \frac{\partial L}{\partial q},  v \right).
$$
\begin{definition}
Given  an induced Dirac structure $D_{\Delta_{Q}}$ on $T^{\ast}Q$,  the equations of motion of a {\bfi Lagrange-Dirac dynamical system} (or an {implicit Lagrangian system})  $(\mathbf{d}_{D} L,D_{\Delta_Q})$ is given by 
\begin{equation}\label{LDirac_system} 
((q(t),p(t),\dot{q}(t),\dot{p}(t)), \mathbf{d}_{D} L(q(t),v(t))) \in D_{ \Delta _Q }(q(t),p(t)),
\end{equation} 
where $t \in [t_{1},t_{2}]$ denotes the time and we denote by $\dot{q}(t)$ and $\dot{p}(t)$ the time derivatives of $q(t)$ and $p(t)$.
\end{definition}
\paragraph{Remark.} It follows from equation \eqref{LDirac_system} that the equality condition for the base points, which corresponds exactly to the Legendre transform $p = \partial L/ \partial v$, automatically is satisfied.
\medskip

Any curve $(q(t),v(t),p(t)) \in TQ \oplus T^{\ast}Q$ satisfying \eqref{LDirac_system} is called a {\bfi solution curve} of the implicit Lagrangian system.

\paragraph{Local Expressions.} 
It follows from equations \eqref{localdirac} and \eqref{LDirac_system} that the Lagrange-Dirac dynamical system may be locally given by
\begin{equation*} 
p = \frac{\partial L}{\partial v }, \quad \dot{q} =v \in \Delta_{Q} (q), \quad  \dot{p} - \frac{\partial L}{\partial q}
\in \Delta_{Q} ^{\circ} (q).
\end{equation*}
For the unconstrained case, $\Delta_{Q}=TQ$, we can develop the equations of motion called {\bfi implicit Euler-Lagrange equations:}
\begin{equation*}
\begin{split}
p=\frac{\partial L}{\partial v},  \quad \dot{q} =v, \quad   \dot{p} =\frac{\partial L}{\partial q}.
\end{split}
\end{equation*}
Note that the implicit Euler--Lagrange equation contains the Euler--Lagrange equation $ \dot{p}={\partial L}/{\partial q}$, the Legendre transformation,  $p={\partial L}/{\partial v}$, and the second-order condition, $\dot{q}=v$.  In summary, the implicit Euler--Langrange equation provides an DAE on $TQ \oplus T^{\ast}Q$ which is capable of handling degenerate Lagrangians, while the original Euler--Lagrange equation is a second order ODE on $Q$.

\paragraph{The Hamilton-Pontryagin Principle.} 
As is well known, for unconstrained mechanical systems, a solution curve $q(t) \in Q$ of the Euler-Lagrange equation satisfies Hamilton's principle:
\[
  	\delta \int_{t_{1}}^{t_{2}}{L(q(t),\dot{q}(t)) dt} = 0,
\]
for arbitrary variations $\delta q(t) \in TQ$ with fixed endpoints. However, for the case of a degenerate Lagrangian $L$ and with a constraint distribution $\Delta_Q \subset TQ$,  we prefer to employ variational principles on $TQ \oplus T^{\ast}Q$ since primary constraint sets associated to the degenerate Lagrangians and $\Delta_{Q}$ may be given as a subset of $TQ \oplus T^{\ast}Q$ \cite{Dirac1950,YoMa2006a}. So, the natural choice is the {\it Hamilton-Pontryagin principle}, which is given by the stationary condition for curves $(q(t),v(t),p(t)), \, t \in [t_{1}, t_{2}]$ in  $TQ \oplus T^{\ast}Q$ denotes:
\begin{equation*}
	\delta \int_{t_{1}}^{t_{2}} L(q(t),v(t)) + \left\langle p(t),  \dot{q}(t)-v(t) \right\rangle \, dt =0
\end{equation*}
for variations $\delta q(t) \in \Delta_Q$ with fixed endpoints and arbitrary fiberwise variations $\delta p(t)$ and $\delta v(t)$.  

\paragraph{Example:  Harmonic Oscillators.}
  Here we will derive a Lagrange-Dirac dynamical system associated to a linear harmonic oscillator.  In this case, the configuration space is $Q = \mathbb{R}$ where $q \in Q$ represents the position of a particle on the real line.  The Lagrangian is given by $L(q,v) = {v^2}/{2}  - {q^{2}}/{2}$.  Recall that the canonical Dirac structure on $T^{\ast}Q$ is given by $D=\text{graph}(\Omega^\flat)$. 
  \medskip

The Lagrange-Dirac dynamical system $(\mathbf{d}_{D}L,D)$ satisfies, for each $(q,v, p) \in TQ \oplus T^{\ast}Q$,
 \begin{align*}
 ( (q,p,\dot{q},\dot{p}), \mathbf{d}_{D}L(q,v)) \in D(q,p), 
 \end{align*}
 where $p=\partial{L}/\partial{v}$ holds.  It immediately follows $\mathbf{d}_{D}L(q,v) = \Omega^\flat (q,p)\cdot (\dot{q},\dot{p})$.  In local coordinates we may write $\mathbf{d}_{D}L(q,v) =  v dp + q dq$ and $\Omega^\flat(q,p) (\dot{q}, \dot{p}) = -\dot{p} dq + \dot{q} dp$.  Thus, the dynamics of harmonic oscillators may be given by the equations: 
 \[
	 \quad \dot{q} = v, \quad \dot{p} = -q, \quad p=v.
 \]
 
\paragraph{Lagrange-Dirac Systems with External Forces.}
One can lift an external force field  $F: TQ \to T^{\ast}Q$, to a map $\widetilde{F}: TQ \to T^{\ast} T^{\ast} Q$ by the formula
\[
	\langle \widetilde{F}(q,v) , w \rangle = \left< F(q,v), T\pi_{Q}(w) \right> \text{ for all } w \in TT^{\ast}Q,
\]
 Locally, $\widetilde{F}$ is given by $\widetilde{F}(q,v)=(q,p, F(q,v) ,0)$ \cite[\S 7.8]{MandS}.
\medskip

Given a Lagrangian $L:TQ \to \mathbb{R}$ (possibly degenerate), the equations of motion for a {\bfi Lagrange-Dirac system with an external force field} 
$(\mathbf{d}_{D} L, F,D_{\Delta_Q})$ are given by 
\begin{equation*}\label{implLag}
((q(t),p(t),\dot{q}(t),\dot{p}(t)), \mathbf{d}_{D} L(q(t),v(t))-\widetilde{F}(q(t),v(t))) \in D_{ \Delta _Q }(q(t),p(t)).
\end{equation*} 
It follows that the dynamics may be described in local coorindes by
\begin{equation}\label{IPCL_Local}
\dot{q}=v \in \Delta_{Q}(q), \quad  \dot{p}-\frac{\partial L}{\partial q}-F \in \Delta_{Q}^{\circ}(q), \quad  p=\frac{\partial L}{\partial v}.
\end{equation}

Any curve $(q(t),v(t),p(t)) \in TQ \oplus T^{\ast}Q,\,t \in [t_1, t_2]$ is a {\bfi solution curve} of $( \mathbf{d}_{D} L, F,D_{\Delta_Q})$ if and only if it satisfies \eqref{implLag}.

\paragraph{Power Balance Law.} Let $E_{L}(q,v,p)=\left<p, v \right>-L(q,v)$ be a generalized energy on $TQ\oplus T^{\ast}Q$. A solution curve $(q(t),v(t),p(t))$ of $( \mathbf{d}_{D} L, F,D_{\Delta_Q})$ satisfies the power balance condition:
\begin{align*}
\frac{d}{dt}E_{L}(q(t),v(t),p(t))=\left<F(q(t),v(t)), \dot{q}(t) \right>,
\end{align*}
where $\dot{q}(t)=v(t) \in \Delta_{Q}(q)$ and $p(t)=(\partial{L}/\partial{v})(t)$.
\paragraph{The Lagrange-d'Alembert-Pontryagin Principle.} 
Now, we explore the variational structures for Lagrange-Dirac systems with external force fields.  The Lagrange-d'Alembert-Pontryagin principle (or LDAP principle) for a curve $(q(t),v(t),p(t))$, $t \in [t_1, t_2],$  in $TQ \oplus T^{\ast}Q$ is given by
\begin{equation*}\label{LagDAPontPrin_Force}
\begin{split}
&\delta \int_{t_1}^{t_2}  L(q(t),v(t)) + \left<p(t), \dot{q}(t)-v(t) \right> \;dt + \int_{t_1}^{t_2} \left<F(q(t),v(t)), \delta{q}(t)\right>\,dt 
=0
\end{split}
\end{equation*}
for variations $\delta{q}(t) \in \Delta_{Q}(q(t))$ with the endpoints fixed and for all variations of $v(t)$ and $p(t)$, together with the constraint $\dot{q}(t) \in  \Delta_{Q}(q(t))$.
\begin{proposition}
  A curve in $TQ \oplus T^{\ast}Q$ satisfies the LDAP principle if and only if it satisfies \eqref{IPCL_Local}.
\end{proposition}

\begin{proof}
  Taking an appropriate variation of $q(t),v(t)$ and $p(t)$ with fixed end points yields:
\[
\int_{t_{1}}^{t_{2}}\left< \frac{\partial L}{\partial q}-\dot{p} +F, \delta{q}\right> +\left< \frac{\partial
L}{\partial{v}} -p, \delta{v} \right>+ \left<\delta{p}, \dot{q}-v \right>dt=0,
\]
  This is satisfied for all variations $\delta{q}(t) \in \Delta_{Q}(q(t))$ and arbitrary variations $\delta{v}(t)$ and $\delta{p}(t)$, and with the constraint $\dot{q}(t) \in \Delta_{Q}(q(t))$ if and only if \eqref{IPCL_Local} is satisfied.
\end{proof}

\paragraph{Coordinate Expressions.} 
The constraint set $\Delta_Q$ defines a subspace on each fiber of $TQ$, which can be locally be expressed as a subset of $\mathbb{R}^n$. If the dimension of $\Delta_Q(q)$ is $n-m$, then we can choose a basis $e _{m+1}(q), e _{m+2}(q),\ldots, e _n (q)$ of $\Delta(q)$. Recall that the constraint set can be also represented by the annihilator $\Delta^{\circ}(q)$, which is spanned by $m$  one-forms $\omega^{1}, \omega^{2}, \ldots, \omega^{m}$ on $Q$. It follows that equation \eqref{IPCL_Local} can be represented, in coordinates, by employing the Lagrange multipliers $\mu_{a},\,a=1,...,m$, as follows:
\begin{align*}
\begin{pmatrix}
\dot{q}^{i} \\
\dot{p}_{i} \\
\end{pmatrix}
& =
\begin{pmatrix}
0 & 1\\
-1 & 0
\end{pmatrix}
\begin{pmatrix}
-\frac{\partial{L}}{\partial{q}^{i}}-F_{i} \vspace{1mm} \\
v^{i}
\end{pmatrix}
+
\begin{pmatrix}
0 \\
\mu_{a}\,\omega_{i}^{a}
\end{pmatrix}, \\
p_{i}&=\frac{\partial L}{\partial v^{i}},\\
0&=\omega^a_{i}\,v^{i},
\end{align*}
where we employ the local expression $\omega^{a}=\omega^{a}_{i}\,dq^{i}$.

\paragraph{Example: Harmonic Oscillators with Damping.} As before, let $Q = \mathbb{R}$, $L(q,v) = {v^2}/{2} - {q^2}/{2}$ and $D= \operatorname{graph} \, \Omega^\flat$.  Now consider the force field $F:TQ  \to T^{\ast}Q$ defined by $F(q,v) = -(r v) dq$, where $r$ is a positive damping coefficient. Then, $\widetilde{F}(q,v) = (q,p,rv,0)$.  The formulas in equation \eqref{IPCL_Local} give us the equations:
\begin{align*}
	\dot{q} = v, \quad \dot{p} + q + rv = 0,
\end{align*}
with the Legendre transformation $p = v$.

\section{Tensor Products of Dirac Structures}
\label{sec:tensor_products}
\paragraph{Tearing and Interconnecting Physical Systems.}
For modeling complicated physical systems such as multibody systems, large scale networks, electromechanical systems and molecular systems, it is quite useful to employ a {\it modular decomposition}; one may {\it decompose or tear} the concerned system into separate constituent subsystems and then reconstruct the whole system by interconnecting the separate subsystems. In particular, the interconnection may be regarded as  a {\it power conserving interaction} in a variety of ways. Such a power conserving interaction may be physically appeared, for instance, as massless hinges, soldering of wires, conversion of current into torque by a motor, interaction potentials, etc. 

In this section, we show that many power conserving interactions can be effectively expressed by Dirac structures.  A typical interaction between two separate physical system is illustrated in Figure \ref{MechConnect}. We assume that the interaction between two particles holds the power invariance 
$$
\left<f_{1}(t), v_{1}(t)\right>+\left<f_{2}(t), v_{2}(t)\right>=0
$$
for all time $t \in [t_{1},t_{2}]$, such that the velocities, $v_{i}$, and forces, $f_{i}$, satisfy the condition
$$
((v_{1},v_{2}),(f_{1},f_{2})) \in \Sigma_{Q} \times \Sigma_{Q}^{\circ},
$$
where $\Sigma_{Q}$ is a given distribution associated to the interaction.
This does not determine the forces $f_1$ and $f_2$, but instead constrains the set of admissible forces.
If one models the forces using an interaction potential, this places an admissibility constraint on such a potential (e.g. the potential may only depend on the distances between the particles).
\begin{figure}[h]
\begin{center}
\vspace{5mm}
\includegraphics[scale=.45]{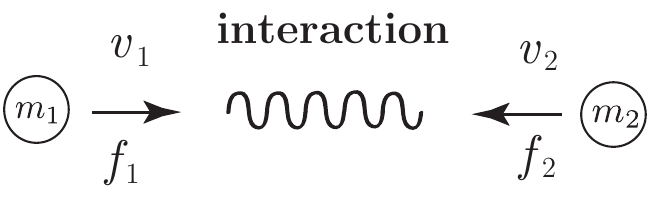}
\caption{Interaction between Two Particles}
\label{MechConnect}
\end{center}
\end{figure}

In this section, we will show how such an {\it interaction Dirac structure} $D_{\mathrm{int}}$ is constructed from $\Sigma_{Q}$.  We will then show how separate systems with Dirac structures $D_1, \dots, D_n$ can be interconnected by $D_{\mathrm{int}}$.
   At this point the readers might justifiably ask ``why one would use interaction Dirac structures to interconnect systems ?'' The answer is that the Dirac structure of an interconnected dynamical system can be expressed by
   \[
   	\underbrace{ D }_{ \begin{array}{c} \text{ \small interconnected} \\ \text{\small Dirac structure} \end{array} } \!\!\!\!=\!\!\!\! \overbrace{ ( D_1 \oplus \cdots \oplus  D_n )}^{ \text{separate Dirac structures} }  \!\!\!\! \!\!\!\! \underbrace{ \boxtimes }_{ \begin{array}{c} \text{ \small tensor } \\ \text{ \small product } \end{array} } \!\!\!\!\!\!\!\! \overbrace{ D_{\text{int}}}^{\text{ interaction }}\!\!\!\!\!,
   \]
Where $\boxtimes$ is a tensor product which will be explained in the sequel.  Such an expression is quite useful for the purpose of the modular modeling as it allows one to describe systems in isolation before discussing the couplings between them.  We refer to the transition from the separate Dirac structures $D_1, \dots, D_n$ to the interconnected Dirac structure $D$ as an {\it interconnection of Dirac structures}.

\paragraph{Standard Interaction Dirac Structures.}
Consider a regular distribution $\Sigma_{Q} \subset TQ$ and define the lifted distribution on $T^{\ast}Q$ by
\[
	\Sigma_{\mathrm{int}}=(T\pi_{Q})^{-1}(\Sigma_{Q}) \subset TT^{\ast}Q.
\]
Let $\Sigma_{\mathrm{int}}^{\circ}$ be the annihilator of $\Sigma_{\mathrm{int}}$. Then, a {\bfi standard interaction Dirac structure} on $T^{\ast}Q$ is defined by, for each $(q,p) \in T^{\ast}Q$,
\begin{align*}
D_{\mathrm{int}}= \Sigma_{\mathrm{int}} \oplus  \Sigma_{\mathrm{int}}^{\circ}.
\end{align*}

Alternatively, one can formulate $D_{\mathrm{int}}$ by using the Dirac structure $D_Q = \Sigma_Q \oplus \Sigma^\circ_Q$ and set
\begin{align}\label{IntDirac}
D_{\rm int}=\pi_{Q}^{\ast}D_{Q}.
\end{align}
 In the next example we will see how this Dirac structure implies {\bfi Newton's third law of action and reaction} (see \cite{YoMa2006a}).
   
\paragraph{Example: Two Particles Moving in Contact.}
Consider two masses on the real line which are constrained to remain in contact.  Denote the velocities velocities are given by $(v_{1},v_{2}) \in V = \mathbb{R}^2$, where $v_{i}$ denotes the velocity of the $i$-th particle.  Since the two particles are in contact and their velocities are common, it follows
\begin{align*}
	(v_{1}, v_{2}) \in \Sigma_{V} \subset V, 
\end{align*}
where $\Sigma_{V}=\{ (v_{1}, v_{2}) \mid v_{1} = v_{2}\}$ is a constraint subspace of $V$.  This constraint is enforced through the associated constraint forces $(f_{1},f_{2}) \in V^{\ast}$ at the contact point, where $f_{i}$ denotes the velocity of the $i$-th particle.  In particular, $F$ must satisfy the constraint
\begin{align*}
	(f_1, f_2) \in \Sigma_{V}^{\circ} \subset V^{\ast},
\end{align*}
where $\Sigma_{V}^{\circ} =\{ (f_{1}, f_{2}) \mid f_{1} = -f_{2}\}$ is the annihilator of $\Sigma_{V}$.  This is the content of Newton's third law, ``every action has an equal and opposite reaction''.  Finally we can define the interaction Dirac structure as
$$
D_{V} = \Sigma_V \oplus \Sigma_V^\circ. 
$$
The two particles moving with the velocities $v_1$ and $v_2$ under the exerting forces $F_1$ and $F_2$ will obey the dynamics of two particle moving in contact if and only if $(v_{1},v_{2},F_{1},F_{2}) \in D_{V}$.  Therefore $D_{V}$ denotes the constraint on tuples of admissible velocities and constraint forces for the system.  Needless to say, one can develop the interaction Dirac structure on $T^*{V} \equiv V \times V^{\ast}$ as well via \eqref{IntDirac}.

\paragraph{Example: Interaction of Two Circuits.}
Consider an interaction between two separate circuits as shown in Figure \ref{interaction_circuit}. Let $Z_{1}$ and $Z_{2}$ denote the impedances, $v_{1},v_{2} \in V$ the currents (usually denoted with $i$'s) and $f_{1},f_{2} \in V^{\ast}$ the voltages (usually denoted with $v$'s) associated to $Z_{1}, Z_{2}$ respectively. The interaction may be simply represented by a {\it two-port} circuit whose 
constitutive relations are given by
$$
v_{1}=v_{2} \;\;\;\textrm{and} \;\;\; f_{1}+f_{2}=0.
$$
These are {\bfi Kirchhoff's laws of currents and voltages}, which clearly correspond to the circuit analogue of Newton's third law. 
In particular the set of admissible currents defines the constraint subspace $\Sigma_V$ and one can construct
the Dirac structure 
$
D_{V} = \Sigma_V \oplus \Sigma_V^\circ.
$
The interaction Dirac structure is $D_{\mathrm{int}} = \pi_{V}^* D_V$.
\begin{figure}[h] 
   \centering
\includegraphics[scale=.7]{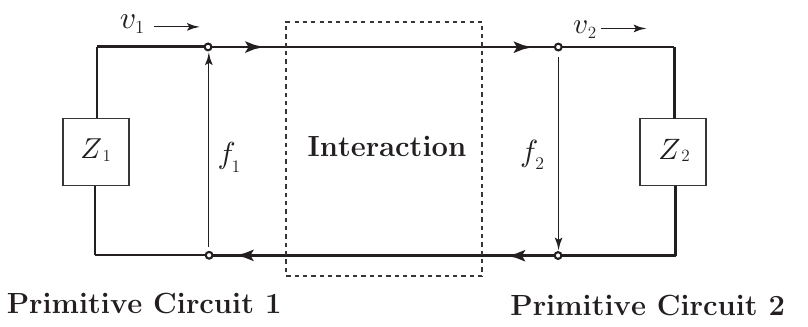} 
   \caption{Interaction of Circuits}
   \label{interaction_circuit}
\end{figure}

\paragraph{Remarks.} In this paper, we will mainly consider interaction Dirac structures of the form $D_{\mathrm{int}}=\Sigma_{\mathrm{int}} \oplus \Sigma^\circ_{\mathrm{int}}$, while there exists a more general class of interaction Dirac structures. For instance, the Lorentz force on a charged particle moving through a magnetic field can be represented by an interaction Dirac structure induced from a {\it magnetic two-form}. It is known that analysis of such a coupled system may be generalized into {\it Lagrangian reduction theory} (see \cite{MandS}). This will need to be the subject of future work.  For now we will provide two examples which hopefully illustrate what is possible.

\paragraph{Example: A Particle Moving Through a Magnetic Field.}
  Consider an electron moving through a vacuum in $Q=\mathbb{R}^{3}$.  The equations of motion are just given by $\ddot{x} = 0, \ddot{y} = 0, \ddot{z} = 0$.  We could think of this system as a set of three decoupled systems with constant velocities.  Now given a magnetic field $\mathbf{B}=B_{x}\mathbf{i}+B_{y}\mathbf{j}+B_{z}\mathbf{k}$ and let $B$ be a closed two-form on $Q=\mathbb{R}^{3}$ defined by
\[
\mathbf{i}_{\mathbf{B}}(dx \wedge dy \wedge dz)=B,
\]
where
\[
B=B_{x} dy \wedge dz +B_{y} dz \wedge dx +B_{z} dx \wedge dy.
\]
Using $B$, one can define a closed  two-form $\Omega_{\mathrm{int}}$ on $T^{\ast}Q=\mathbb{R}^{3} \times \mathbb{R}^{3}$ by $\Omega_{\rm int} =-\frac{e}{c}\pi_{Q}^{\ast} B$.
The force on a charged particle moving through the magnetic field $B$ is given the {\it Lorentz force}, $f=-({e}/{c})\,\mathbf{i}_{\mathbf{v}}B$. In other words, the Lorentz force couples the dynamics of the particle with the magnetic field.
If we desire to express this coupling in the form of an interaction Dirac structure, one could define the {\bfi magnetic Dirac structure} $D_{\mathrm{mag}}$ on $T^{\ast}Q$ by $D_{\text{mag}}=\mathrm{graph}\,\Omega_{\mathrm{\,int}}^{\flat}$.

\paragraph{Example:  An Ideal Direct Current Motor.}
  The form of the Dirac structures given in the previous paragraph also describes the structure of an ideal Direct Current (DC) motor.  In this case, the configuration manifold may given by $\mathbb{R} \times S^1$, where the first component is the charge through the armature of a DC motor and the second component is the angle of the motor shaft, which represents an element of the unit circle.  When an armature (a coil with wiring loops) current $I$ passes through the magnetic field of the motor, it generates a motor torque as $\tau = K \cdot I$ for some motor constant $K$.  Geometrically, given coordinates $(q, \theta)$ on $\mathbb{R} \times S^1$, we can express the relationship between current and torque with the two-form $B = K dq \wedge d \theta$ so that $\tau = B( I , \cdot )$.  Finally, this can all be expressed with the Dirac structure
  \begin{align*}
  	D_{\text{motor}} &= \mathrm{graph}\,B \\
		&= \{ ((I, \omega) , (V, \tau)) \in T(\mathbb{R} \times S^1) \times T^{\ast}(\mathbb{R} \times S^1) \mid V= -K \cdot \omega , \;\; \tau = K \cdot I \},
  \end{align*}
where $I$ and $V$ are the current and voltage associated with the armature of the DC motor, and $\omega$ and $\tau$ are the angular velocity and torque of the motor shaft.  Given a circuit and a mechanical system connected by an ideal motor, the above interaction Dirac structure would characterize the interconnection between electrical and mechanical systems.  This is an early step in understanding an interconnection of electro-mechanical systems in Lagrangian mechanics.

\paragraph{The Direct Sum of Dirac Structures.}
  So far we have shown how to express interconnections as interaction Dirac structures.  We intend to use these interaction Dirac structures to interconnect subsystems on separate manifolds $M_1$ and $M_2$.  However, before going into the interconnection of mechanical systems on separate manifolds, let us formalize the notion of a ``direct sum'' of systems on separate spaces.  Given two vector bundles $V_1 \to M_1$ and $V_2 \to M_2$ the direct sum $V_1 \oplus V_2$ is a vector bundle over $M_1 \times M_2$. In the context of Dirac structures (which are a special case) we have the following additional closures.
\begin{proposition}
\label{prop:direct_sum}
If  $D_1 \in \mathrm{Dir}{(M_1)}, D_2 \in \mathrm{Dir}{(M_2)}$, then  $D_1 \oplus D_2 \in \mathrm{Dir}{(M_1 \times M_2)}$.
Moreover, if $D_1$ and $D_2$ are integrable, then $D_1 \oplus D_2$ is integrable.
\end{proposition}
\begin{proof}
	As the dimension of each fiber of $D_1 \oplus D_2$ is equal to $\dim(M_1) + \dim(M_2)$ is sufficient to prove that $D_1 \oplus D_2$ is isotropic in order to assert that it is a Dirac structures.
	The isotropic condition can be verified taking an arbitrary $(v_1,v_2,\alpha_1,\alpha_2) , (w_1,w_2,\beta_1,\beta_2) \in D_1 \oplus D_2$ and noting
	\[
		\langle \! \langle (v_1,v_2,\alpha_1,\alpha_2) , (w_1,w_2,\beta_1,\beta_2) \rangle \! \rangle= \langle \! \langle (v_1,\alpha_1), (w_1,\beta_1) \rangle \! \rangle + \langle \! \langle (v_2,\alpha_2), (w_2,\beta_2) \rangle \! \rangle = 0
	\]
	where the final equality follows from the isotropy of $D_1$ and $D_2$.
	A similarly simple verification holds for proving integrability by computing the left hand side of \eqref{ClosedCond} and noting the formula splits into a direct sum of two parts which are clearly contained in $D_1$ and $D_2$ respectively  by the assumed integrability of $D_1$ and $D_2$.
\end{proof}

The following corollary is, perhaps, equally obvious.  However, it is particularly relevant for the case at hand.

\begin{corollary}
Let $\Omega_{i}$ be the canonical symplectic structures on $T^{\ast}Q_{i}$ and $D_{\Delta_{Q_{i}}}$ the Dirac structures on $T^{\ast}Q_{i}$ induced from constraint distributions, $\Delta_{Q_i} \subset TQ_i$, for $i=1,2$.
Then $D_{\Delta_{Q_{1}}} \oplus D_{\Delta_{Q_{2}}}$ may be expressed as an induced Dirac structure on $T^{\ast}(Q_1 \times Q_2)$.
In particular, $D_{\Delta_{Q_1} } \oplus D_{\Delta_{Q_2}} = D_{\Delta_{Q_1} \oplus \Delta_{Q_2} }$.
\end{corollary}

It is notable that the direct sum of Dirac structures does not express any interaction between separate systems.
To express interactions using Dirac structures associated to power-conserving couplings we will require a tensor product of Dirac structures.
\begin{definition}[\cite{Gua2011}]
 Let $D_a,D_b \in \mathrm{Dir}(M)$.  Let $d: M \hookrightarrow M \times M$ be the diagonal embedding in $M \times M$. We define the {\bfi Dirac tensor product} of $D_a$ and $D_b$ by
    \begin{align*}
    	D_a \boxtimes D_b := d^{\ast}(D_a \oplus D_b) \equiv \frac{ (D_a \oplus D_b \cap K^\perp) + K }{K},
    \end{align*}
where $K = \{ (0,0) \} \oplus\{(\beta,-\beta)  \} \subset T(M \times M) \oplus T^{\ast}(M \times M)$
and its orthogonal complement $K^{\perp} \subset T(M \times M) \oplus T^{\ast}(M \times M)$ is given by $K^{\perp} = \{ (v,v)\} \oplus  T^{\ast}(M \times M)$.
\end{definition}

\begin{theorem}[\cite{Gua2011}]
	If $D_a \oplus D_b \cap K^\perp$ has locally constant rank then $D_a \boxtimes D_{b}$ is a Dirac structure on $M$.
\end{theorem}
\begin{corollary}
	Let $D_{\Delta_Q}$ be a constraint induced Dirac structure on $T^{\ast}Q$, and let $D_{\rm int} = \pi_Q^* ( \Sigma_Q \oplus \Sigma_Q^\circ)$ be a Dirac structure given by a constraint distribution $\Sigma_Q \subset TQ$.
	Then $D_{\Delta_Q} \boxtimes D_{\rm int}$ is a Dirac structure if $\Delta_Q \cap \Sigma_Q$ is a regular distribution.
\end{corollary}

\paragraph{Remark.}
By the definition of $\boxtimes$, it is clear that if $D_1$ and $D_2$ are integrable Dirac structures, then $D_1 \boxtimes D_2$ is integrable.

\paragraph{Remark.}
In  \cite{YoJaMa2010} and \cite{JYM2010}, we defined the bowtie product
	\begin{align}\label{defbowtie}
		D_a \bowtie D_{b}& = \{ (v,\alpha) \in T{M}\oplus T^{\ast}M \mid \exists \beta \in T^{\ast}M  \nonumber \\
		&\hspace{2cm} \textrm{ such that } (v,\alpha+\beta) \in D_a,(v,-\beta) \in D_{b} \},
	\end{align}
which is equivalent with the tensor product, $\boxtimes$.\footnote{We appreciate Henrique Bursztyn for pointing out this fact in Iberoamerican Meeting on Geometry, Mechanics and Control in honor of Hern\'an Cendra at Centro At\'omico Bariloche, January 13, 2011.}

\paragraph{Properties of the Dirac Tensor Product.}
  It has been shown already that the Dirac tensor product is associative, commutative, and preserves the integreability condition (see \cite{Gua2011}).   Here, we will review these properties with the use a special blinear map, $\Omega_{\Delta_{M}}: \Delta_M \oplus \Delta_M \to \mathbb{R}$, induced from a Dirac structure $D$ on $M$ with $\Delta_{M} = \pr_{TM}(D) \subset TM$, where $\pr_{TM}: TM \oplus T^{\ast}M;\, (v,\alpha) \mapsto v$ and we assume that $\Delta_{M}$ is smooth.

\begin{lemma}
	\label{lem:two-form}
	On each fiber of $ T_{x}M \times T_{x}^{\ast}M$ at $x \in M$, there exists a bilinear anti-symmetric map $\Omega_{\Delta_{M}}(x) : \Delta_{M}(x) \times \Delta_{M}(x)  \to \mathbb{R}$ defined by the property 
	\begin{equation*}
		\Omega_{\Delta_{M}}(x)( v_1, v_2) = \lb \alpha_1, v_2 \rb \;\text{when }\; (v_1,\alpha_1) \in D(x). 
		\label{dirac_2_form}
	\end{equation*}
\end{lemma}

This bilinear map was initially introduced by \cite{CoWe1988} for the case of linear Dirac structures.  We can easily generalize it to the case of general manifolds since $\Omega_{\Delta_{M}}$ may be defined fiberwise (see also \cite {Cour1990} and \cite{DuWa2004}).
\medskip

	Given a Dirac structure $D \in \mathrm{Dir}(M)$, it follows from equation \eqref{DiracManifold} that, for each $x \in M$, $D(x)$ may be given by 
\begin{align*}
D(x)=\{ (v, \alpha) \in T_{x}M \times  T_{x}^{\ast}M
  \;& \mid \;  v \in \Delta_{M}(x),\; \mbox{and} \; \;\\
   &\alpha(w)=\Omega_{\Delta_{M}}(x)(v,w) \; \;
\mbox{for all} \; \; w \in \Delta_{M}(x) \},
\end{align*}

\begin{proposition}
\label{thm:bowtie}
  Let $D_a$ and $D_b \in \mathrm{Dir} (M)$.  Let $\Delta_a = \pr_{TM}(D_a)$ and $\Delta_{b} =\pr_{TM}(D_b)$. Let $\Omega_a$ and $\Omega_{b}$ be the bilinear maps induced by $D_a$ and $D_b$ respectively. If $\Delta_{a} \cap \Delta_{b}$ has locally constant rank, then $D_{a} \boxtimes D_{b}$ is a Dirac structure with the smooth distribution $\pr_{TM}(D_{a} \boxtimes D_{b}) = \Delta_{a} \cap \Delta_{b}$ and with the bilinear map $(\Omega_{a}+\Omega_{b})|_{\Delta_{a} \cap \Delta_{b}}$.
\end{proposition}

\begin{proof}
\label{pf:bowtie}
Let $(v,\alpha) \in D_{a} \boxtimes D_b(x)$ for $x \in M$.  By definition of the Dirac tensor product in \eqref{defbowtie}, there exists $\beta \in T^{\ast}_{x}M$ such that $(v, \alpha + \beta) \in D_{a}(x), (v,-\beta) \in D_b(x)$.  Hence, one has 
\[
  	 \Omega_a^{\flat}(x) \cdot v - \alpha - \beta \in \Delta_a^\circ(x) \;\;\text{and}\;\;
	 \Omega_b^{\flat}(x) \cdot v + \beta \in \Delta_b^\circ(x), \; \text{for each}\; x \in M, 
\]
where $v \in \Delta_a(x)$ and $v \in \Delta_b(x)$. This means $(\Omega_a^{\flat} + \Omega_b^{\flat})(x)\cdot v - \alpha \in \Delta_a^\circ(x) + \Delta_b^\circ(x)$ and $v \in \Delta_a \cap \Delta_b(x)$.  But $\Delta_a^\circ(x) + \Delta_b^\circ(x) = (\Delta_a \cap \Delta_b)^\circ(x)$.  Therefore, upon setting $\Omega_c = \Omega_a + \Omega_b$ and $\Delta_c = \Delta_a \cap \Delta_b$, we can write
$\Omega_c^{\flat}(x)\cdot v - \alpha \in \Delta_c^\circ(x)$ and $v \in \Delta_c(x)$; namely, $(v,\alpha) \in D_c(x)$, where $D_{c}$ is a Dirac structure with $\Delta_{c}$ and $\Omega_{c}$. Then, it follows that $D_{a} \boxtimes D_b \subset D_c$.  Equality follows from the fact that both $D_{a} \boxtimes D_b(x)$ and $D_c(x)$ are subspaces of $T_{x}M \times T^{\ast}_{x} M$ with the same dimension.
\end{proof}

\begin{corollary}
\label{cor:bowtie2}
  If $\Omega_b = 0$, then it follows that $D_b = \Delta_b \oplus \Delta_b^\circ$ and also that $D_c = D_a \boxtimes D_b$ is induced from $\Delta_a \cap \Delta_b$ and $\Omega_a|_{\Delta_a \cap \Delta_b}$.
\end{corollary}

\begin{proposition}\label{associative} 
	Let $D_a,D_b,D_c \in \mathrm{Dir} (M)$ with smooth distributions $\Delta_a = \pr_{TM}(D_a)$, $\Delta_b = \pr_{TM}(D_b)$, and $\Delta_c = \pr_{TM}(D_c)$. Assume that $\Delta_{a} \cap \Delta_{b}$, $\Delta_{b} \cap \Delta_{c}$ and $\Delta_{c} \cap \Delta_{a}$ have locally constant ranks.  Then the Dirac tensor product $\boxtimes$ is associative and commutative; namely we have
	\[
		(D_a \boxtimes D_b) \boxtimes D_c = D_a \boxtimes (D_b \boxtimes D_c)
	\]
	and
	\[
		D_a \boxtimes D_b = D_b \boxtimes D_a.
	\]
\end{proposition}

\begin{proof}
  First we prove commutativity.  Recall that any Dirac structure may be constructed by its associated constraint distribution $\Delta = \pr_{TM}(D)$ and the Dirac two-form $\Omega_\Delta$.  Let $\Omega_a, \Omega_b$,and $\Omega_c$ be the bilinear maps induced by  $D_a,D_b$, and $D_c$ respectively.  Then we find by Proposition \ref{thm:bowtie} that $D_a \boxtimes D_b$ is defined by the smooth distribution $\Delta_{ab} = \Delta_a \cap \Delta_b$ and the bilinear map $\Omega_{\Delta_{ab}} = (\Omega_{\Delta_{a}} + \Omega_{\Delta_{b}})|_{\Delta_{ab}}$.  By commutativity of $+$ and $\cap$, we find the same distribution and the bilinear map for $D_b \boxtimes D_a$, we have $D_a \boxtimes D_b = D_b \boxtimes D_a$.

Next, we prove associativity.  Let
$
	\Delta_{(ab)c} = \pr_{TM}((D_a \boxtimes D_b) \boxtimes D_c)
$
and	
$
\Delta_{a(bc)}=  \pr_{TM}(D_a \boxtimes (D_b \boxtimes D_c))
$
and it follows
\[
  	\Delta_{(ab)c} = (\Delta_a \cap \Delta_b) \cap \Delta_c = \Delta_a \cap (\Delta_b \cap \Delta_c) = \Delta_{a(bc)}.
\]
If $\Omega_{\Delta_{(ab)c}}$ and $\Omega_{\Delta_{a(bc)}}$ are respectively the bilinear maps for $(D_a \boxtimes D_b) \boxtimes D_c$ and $D_a \boxtimes (D_b \boxtimes D_c)$, we find
\begin{align*}
	\Omega_{\Delta_{(ab)c}} &= [(\Omega_{\Delta_{a}} + \Omega_{\Delta_{b}})|_{\Delta_{ab}} + \Omega_{\Delta_{c}} ]|_{\Delta_{(ab)c}} = (\Omega_{\Delta_{a}} + \Omega_{\Delta_{b}} + \Omega_{\Delta_{c}} )|_{\Delta_{(ab)c}} \\
		&= (\Omega_{\Delta_{a}} + \Omega_{\Delta_{b}} + \Omega_{\Delta_{c}}  )|_{\Delta_{a(bc)}} = \Omega_{\Delta_{a(bc)}}.
\end{align*}
  Thus, we obtain
\[
		(D_a \boxtimes D_b) \boxtimes D_c = D_a \boxtimes (D_b \boxtimes D_c).
\]
\end{proof}

\paragraph{Remark.}
  We have shown that the tensor product $\boxtimes$ acts on pairs of Dirac structures with clean intersections to give a new Dirac structure and also that that it is an associative and commutative product.  It is easy to verify that the Dirac structure $D_e = TM \oplus \{0 \}$ satisfies the property of the identity element as $D_e \boxtimes D = D \boxtimes D_e=D$ for every $D \in \mathrm{Dir}(M)$.  However this does \emph{not} make the pair $(\mathrm{Dir}(M), \boxtimes)$ into a commutative category because $\boxtimes$ is not defined on all pairs of Dirac structures.  This is similar to the difficulty of defining a symplectic category (see \cite{Wein2009}).
\medskip

The previous propositions justify the following definition for the ``interconnection'' of Dirac structures
\begin{definition}
	 Let $(D_1,M_1)$ and $(D_2,M_2)$ be Dirac manifolds and let $D_{\mathrm{int}} \in \mathrm{Dir}(M_1 \times M_2)$ be such that $D_{\mathrm{int}}$ and $D_1 \oplus D_2$ have clean intersections.  Then we define the \emph{interconnection} of $D_1$ and $D_2$ through $D_{\mathrm{int}}$ by the tensor product:
	 \[
	 	(D_1 \oplus D_2) \boxtimes D_{\mathrm{int}}.
	 \]
\end{definition}

\paragraph{Interconnections of Induced Dirac Structures.}
Let $Q_{1}$ and $Q_{2}$ be distinct configuration manifolds and let $D_{\Delta_{Q_1}} \in \mathrm{Dir}(T^{\ast}Q_{1})$ and $D_{\Delta_{Q_2}}  \in \mathrm{Dir}(T^{\ast}Q_{2})$ be Dirac structures induced from smooth distributions $\Delta_{Q_{1}} \subset TQ_1$ and $\Delta_{Q_{2}} \subset TQ_2$. Given a smooth distribution $\Sigma_{Q}$ on $Q=Q_{1} \times Q_{2}$, let $\Sigma_{\mathrm{int}}=(T\pi_{Q})^{-1}(\Sigma_{Q})$ and define $D_{\mathrm{int}} = \Sigma_{\mathrm{int}} \oplus \Sigma_{\mathrm{int}}^\circ$. Then it is clear that $D_{\Delta_{Q_1}}  \oplus D_{\Delta_{Q_2}}$ and $D_{\mathrm{int}}$ intersect cleanly if and only if $\Delta_{Q_1} \oplus \Delta_{Q_2}$ and $\Sigma_{Q}$ intersect cleanly. 

\begin{proposition}
 If $\Delta_{Q_1} \oplus \Delta_{Q_2}$ and $\Sigma_{Q}$ intersect cleanly, then the interconnection of $D_{\Delta_{Q_1}}$ and $D_{\Delta_{Q_2}}$ through $D_{\mathrm{int}}$ is locally given by the {\it Dirac structure induced from $(\Delta_{Q_1} \oplus \Delta_{Q_2})  \cap \Sigma_Q$} as, for each $(q,p) \in T^{\ast}Q$,
\begin{align}\label{InducedDirac}
(D_{\Delta_{Q_1}}  \oplus D_{\Delta_{Q_2}}) &\boxtimes D_{\mathrm{int}}(q,p) =\{\, (w, \alpha)
\in T_{(q,p)}T^{\ast}Q  \times T^{\ast}_{(q,p)}T^{\ast}Q \, \mid \,   \nonumber \\
& \hspace{2cm}
w \in \Delta_{T^{\ast}Q}(q,p)\;\;\mbox{and} \;\;
\alpha-\Omega^{\flat}(q,p) \cdot w
\in\Delta_{T^{\ast}Q}^{\circ}(q,p) \, \},
\end{align} 
where $\Delta_{T^{\ast}Q}=T\pi_{Q}^{-1}((\Delta_{Q_1} \oplus \Delta_{Q_2})  \cap \Sigma_Q)$ and $\Omega=\Omega_{1} \oplus \Omega_{2}$, where $\Omega_{1}$ and $\Omega_{2}$ are the canonical symplectic structures on $T^{\ast}Q_{1}$ and $T^{\ast}Q_{2}$.
\end{proposition}
\begin{proof}
It is easily checked from Corollary \ref{cor:bowtie2}.
\end{proof}
 It is simple to generalize the preceding constructions to the interconnection of $n$ distinct Dirac structures, $D_1, \dots, D_n$, on distinct manifolds, $M_1, \dots, M_n$.
 Specifically, by choosing an appropriate interaction Dirac structure, $D_\mathrm{int} \in \mathrm{Dir}(M_1 \times \dots \times M_n)$, we can define the interconnection of $D_{1}, \dots, D_{n}$ through $D_{\mathrm{int}}$ by the Dirac structure
\begin{equation*}\label{InterconnectDirac}
	D=  \left(\bigoplus_{i=1}^{n}{D_i} \right) \boxtimes D_\mathrm{int}.
\end{equation*}

\paragraph{The Link Between Composition and Interconnection of Dirac Structures.}
The notion of \emph{composition} of Dirac structures was introduced in \cite{CeVdeScBa2007} in the context of port-Hamiltonian systems, where the composition was constructed on vector spaces. 
 Let $V_1,V_2$ and $V_s$ be vector spaces.  Let $D_1$ be a linear Dirac structure on $V_1 \oplus V_s$ and $D_2$ be a linear
  Dirac structure on $V_s \oplus V_2$.  The \emph{composition} of $D_1$ and $D_2$ is given by
\begin{align*}
	& D_1 || D_2 = \{ (v_1,v_2,\alpha_1,\alpha_2) \in (V_1 \times V_2) \oplus (V_{1}^{\ast} \times V_{2}^{\ast}) \mid \\
		& \hspace{1cm} \exists (v_s,\alpha_s) \in V_s \oplus V_{s}^{\ast},\;
		\text{such that } (v_1,v_s,\alpha_1,\alpha_s) \in D_1, 
		 (-v_s,v_2,\alpha_s,\alpha_s) \in D_2 \},
\end{align*}
where $V_{1}^{\ast}, V_{2}^{\ast}$ and $V_{s}^{\ast}$ denote the dual space of $V_{1}, V_{2}$ and $V_{s}$. It was also shown that the set $D_1 || D_2$ is itself a Dirac structure on $V_{1} \times V_{2}$, and moreover given shared variables the operation of composition is associative.  However the type of interaction given by composition of Dirac structures is specifically the interaction between systems which have shared variables.  The next theorem shows the link between the notion of composition of Dirac structures and the notion of interconnection of Dirac structures.

\begin{proposition}
\label{thm:composition}
 Set $V = V_1 \times V_s \times V_s \times V_2$ and $\bar{V} = V_1 \times V_2$.  Let $\Psi : V \to \bar{V}$ be the projection $(v_1,v_s,v_s',v_2) \mapsto (v_1,v_2)$.  
Let $\Sigma_{\mathrm{int}} = \{ (v_1,v_s,-v_s,v_2)  \in V \}$ and let $D_{\mathrm{int}} = \Sigma_{\mathrm{int}} \oplus \Sigma_{\mathrm{int}}^\circ$.  For linear Dirac structures $D_1$ on $V_1 \times V_s$ and $D_2$ on $V_s \times V_2$, it follows that
\[
	D_1 || D_2 = \Psi_*(D_1 \oplus D_2) \boxtimes D_{\mathrm{int}}.
\]
\end{proposition}
\noindent 
For the details and the relevant proofs, see \cite{JY2011}.

\section{Interconnection of Implicit Lagrangian Systems}
\label{sec:interconnection}
\paragraph{Modular Decomposition of Physical Systems.}
For design and analysis of complicated mechanical systems, one often decomposes the concerned system into several constituent subsystems so that one can easily understand the whole system as an interconnected system of subsystems. 

In this section, we shall show how a Dirac-Lagrange system can be reconstructed as an interconnected system of torn-apart subsystems through an interaction Dirac structure. 
\smallskip

First recall that given a Lagrangian $L : TQ \to \mathbb{R}$ with a smooth distribution $\Delta_{Q}$ on a configuration manifold $Q$, a Lagrange-Dirac dynamical system $( \mathbf{d}_{D} L,D_{\Delta_{Q}})$ that satisfies the condition
\[
((q(t),p(t),\dot{q}(t),\dot{p}(t)), \mathbf{d}_{D} L(q(t),v(t))) \in D_{ \Delta _Q }(q(t),p(t)),
\]
induces the implicit Lagrange-d'Alembert equations:
$$
\dot{q}=v \in \Delta_{Q}(q), \quad  \dot{p}-\frac{\partial L}{\partial q} \in \Delta_{Q}^{\circ}(q), \quad  p=\frac{\partial L}{\partial v}.
$$
Next, decompose the original system into separate subsystems such that 
$$
Q=Q_{1} \times \cdots \times Q_{n} \quad \text{and} \quad L=\sum_{i=1}^{n}L_{i}: TQ \to \mathbb{R}, 
$$
where $L_i:TQ_i \to \mathbb{R},\;i = 1,\dots,n$ are Lagrangians for separate subsystems.  In particular, we can decompose the original system into subsystems in such a way that the distribution $\Delta_{Q}$ can be expressed by
\begin{equation*}\label{DecomDist}
\Delta_{Q}=(\Delta_{Q_1} \oplus \cdots \oplus \Delta_{Q_n})  \cap \Sigma_Q,
\end{equation*}
where $\Delta_{Q_{i}} \subset TQ_{i}$ are smooth constraint distributions for subsystems and $\Sigma_Q \subset TQ$ denotes some constraint distribution due to the interactions at the {\it boundaries} between subsystems.

\paragraph{Tearing into Primitive Subsystems.} 
In the above modular decomposition, we assume that the intersection $(\Delta_{Q_1} \oplus \cdots \oplus \Delta_{Q_n})  \cap \Sigma_Q$ is {\it clean}, namely, the rank of $\Sigma_Q$ is locally constant. 

Without the interaction constraint $\Sigma_{Q}$, the separate subsystems maybe regarded as a set of totally torn-apart systems, each of which is called a {\bfi primitive subsystem}.
\begin{definition}
   Let $Q = Q_1 \times \dots \times Q_n$ and $L_i : TQ_i \to \mathbb{R}$.  For Dirac structures $D_{\Delta_{Q_i} } \in \mathrm{Dir}( T^{\ast}Q_i)$ and interaction forces $F_i : TQ \to T^{\ast}Q_i$, we call each triple $(D_{\Delta_{Q_i}} , \mathbf{d}_{D} L, F_i)$ a \emph{primitive Lagrange-Dirac system} for $i = 1,\dots, n$.  We call the equations of motion given by the condition
\begin{equation*}
\left( (q_{i},p_{i},\dot{q}_{i},\dot{p}_{i}), \mathbf{d}_{D} L_{i}(q_{i},v_{i})-\pi_{Q_{i}}^{\ast}F_{i}(q,v) \right) \in D_{ \Delta _{Q_{i}} }\left( q_{i},p_{i} \right),
\end{equation*} 
the \emph{primitive Lagrange-d'Alembert equations}.
\end{definition}

The primitive Lagrange-d'Alembert equations are locally given by
\begin{equation}\label{SubILS}
\dot{q}_{i}=v_{i} \in \Delta_{Q_{i}}(q_{i}), \quad  \dot{p}_{i}-\frac{\partial L_{i}}{\partial q_{i}}-F_{i} \in \Delta_{Q_{i}}^{\circ}(q_{i}), \quad  p_{i}=\frac{\partial L_{i}}{\partial v_{i}},
\end{equation}

Note that equations of motion in \eqref{SubILS} are {\it not} equivalent to the equations for the original system $(\mathbf{d}_{D}L,D_{\Delta_{Q}})$ unless we know how to explicitly choose the correct interaction force $F$.
As the appropriate force $F$ which produces the dynamics of the interconnected system is usually only defined implicitly (e.g. as a Lagrange multiplier of a constraint) equation \eqref{SubILS} is usually not available to us.
In fact, for a fixed $i$ the primitive Lagrange-d'Alembert equations are not well defined unless one is given the velocities of all of the other systems.  In other words, when we reconstruct the original system $(\mathbf{d}_{D}L,D_{\Delta_{Q}})$ from the torn-apart primitive Lagrange Dirac systems $(\mathbf{d}_{D} L_{i}, F_{i}, D_{\Delta_{Q_{i}}})$, which will be later given by a interaction Dirac structure, we shall need to impose extra constraints on the velocities and forces at the boundaries between the primitive systems. In the following, we shall show such constraints can be given by an interaction Dirac structure. 

\paragraph{Interaction Forces.} 
Before going into details on the interconnection of subsystems, we define a {\bfi total interaction force field} $F=(F_{1},..., F_{n}): TQ \to T^{\ast}Q$ given by interaction forces $F_i : TQ \to T^{\ast}Q_i$ such that the {\bfi power invariance} through the interacting boundaries between the subsystems holds.  If the interconnection structure is given by a constraint distribution $\Sigma_Q$ then the forces must satisfy,
$$
\left< F(q,v), v\right>= \sum_{i=1}^{n}\left< F_{i}(q,v), v_{i}\right> =0.
$$ 
for each $v \in \Sigma_Q$.
In other words, $F(q,v) \in \Sigma_{Q}^{\circ}(q)$ where $\Sigma_{Q}^{\circ}$ is the annihilator of $\Sigma_{Q}$.

In the next section we will consider dynamics which evolve on the phase space $T^{\ast}Q$ so that the forces occur on the iterated cotangent bundle $T^{\ast} T^{\ast}Q$.
In order to accommodate this larger space, recall that a force field $F: TQ \to T^{\ast}Q$ induces a horizontal lift as, for each $(q,v) \in TQ$,
\[
\pi_{Q}^{\ast}F(q,v) \cdot w=\left< F(q,v), T\pi_{Q}(w) \right> \text{ for all } w \in TT^{\ast}Q,
\]
where the horizontal lift  $\pi_{Q}^{\ast}F(q,v)$  
is locally given by $\pi_{Q}^{\ast}F(q,v)=(q,p,F_{v_{q}},0) \in T^{\ast}_{(q,p)}(T^{\ast}Q)$.

 \paragraph{Interconnection of Dirac Structures.} 
In order to formulate the original physical system as an interconnected system, one needs to connect each Dirac structure, $D_{\Delta_{Q_{i}}}$, through the interaction Dirac structure $D_{\text{int}}$. In particular, if $D_{\rm int}$ is defined from a smooth distribution $ \Sigma_{Q}$, we recall that the interaction Dirac structure may be given by, as in \eqref{IntDirac},
$$
D_{\mathrm{int}}=  \pi_{Q}^{\ast}D_{Q}=\pi_{Q}^{\ast}(\Sigma_{Q} \times \Sigma_{Q}^{\circ}).
$$
Recall from equation \eqref{InducedDirac} that the interconnection of separate Dirac structures is given  through the interaction Dirac structure by 
\begin{align*}
D_{\Delta_{Q}}:=(D_{\Delta_{Q_{1}}} \oplus \cdots \oplus D_{\Delta_{Q_{n}}})  \boxtimes D_{\mathrm{int}}.
\label{eq:interconnected_dirac}
\end{align*}
\paragraph{Interconnection of Primitive Lagrange-Dirac Systems.}
We will consider the process of interconnecting separate Dirac structures, which allows us to couple the dynamics of primitive subsystems via the interaction Dirac structure. 

\begin{definition}
Let $(\mathbf{d}_{D} L_{i}, F_{i},D_{\Delta_{Q_{i}}})$ be $n$ distinct Lagrange-Dirac dynamical systems for $i=1,...,n$. Given a smooth distribution $\Sigma_{Q}$ on $Q = Q_1 \times \cdots \times Q_n$, the {\bfi interconnection of primitive Lagrange-Dirac systems} $(\mathbf{d}_{D} L_{i}, F_{i}, D_{\Delta_{Q_{i}}})$ is given by, for $i = 1,\dots n$, 
$$
\left( (q_{i},p_{i},\dot{q}_{i},\dot{p}_{i}) , \mathbf{d}_{D} L_{i}( q_{i},v_{i} )-\pi_{Q_{i}}^{\ast}F_{i}(q,v) \right) \in D_{ \Delta _{Q_{i}} }\left( q_{i},p_{i} \right),
$$
together with the interaction constraints 
$$
\left( (\dot{q}_{1},...,\dot{q}_{n}), \left(F_{1}(q,v), \cdots, F_{n}(q,v \right)\right) \in D_{Q}(q_{1},...,q_{n}).
$$
\end{definition}

\begin{proposition}
\label{prop:LDDS_splitting}
The following statements are equivalent:
  \begin{enumerate}
        \item[(i)] The curves $(q_{i},v_{i},p_{i}) \in TQ_{i} \oplus T^{\ast}Q_{i}$ satisfy
	\begin{align*}
((q_{i},p_{i},\dot{q}_{i},\dot{p}_{i}), \mathbf{d}_{D} L_{i}(q_{i},v_{i})-\pi_{Q_{i}}^{\ast}F_{i}(q,v)) \in D_{ \Delta _{Q_{i}} }(q_{i},p_{i}),
	\end{align*}
        for $i = 1,\dots n$, together with the constraints 
$$
\left(\dot{q}_{1},...,\dot{q}_{n}), \left(F_{1}(q,v), \cdots, F_{n}(q,v\right)\right) \in D_{Q}(q_{1},...,q_{n}).
$$
%
        \item[(ii)] The curve $(q,v,p) \in TQ \oplus T^{\ast}Q$ satisfies
	\begin{align*}
((q,p,\dot{q},\dot{p}), \mathbf{d}_{D} L(q,v)) \in D_{ \Delta _Q }(q,p).
	\end{align*}
\end{enumerate}
\end{proposition}

\begin{proof}
  Assuming (i), one can obtain the Lagrange-Dirac dynamical system as
\[
\left(\left(\dot{q},\dot{p}\right), \left(-\frac{\partial L}{\partial q}, v \right)\right) \in D_{\Delta_{Q}}(q,p),
\]
while it follows from $D_{\Delta_{Q}}=(D_{\Delta_{Q_{1}}} \oplus \cdots \oplus D_{\Delta_{Q_{n}}})  \boxtimes D_{\mathrm{int}}$ that there may exist some $\alpha=(\alpha_q,\alpha_p) \in T^*T^*Q$ such that
\begin{align}
\left(\left(\dot{q}, \dot{p}\right), \left(\alpha_q,\alpha_p \right)\right) \in D_{\mathrm{int}} \label{eq:hoj1}
\end{align}
and
\begin{align}
\left(\left(\dot{q},\dot{p}\right),\left(- \pder{L}{q}-\alpha_q,  v-\alpha_p\right)\right) \in D_{\Delta_{Q_{1}}} \oplus \cdots \oplus D_{\Delta_{Q_{n}}}. \label{eq:hoj2}
\end{align}
Equation \eqref{eq:hoj1} implies $\dot{q} \in \Sigma_{Q}(q), \alpha_q \in \Sigma_{Q}^\circ(q)$ and $\alpha_p = 0$.  This means that $\alpha$ is the horizontal lift of $F(q,v) = \alpha_q$.  The interaction forces $F(q,v) \in \Sigma_{Q}^{\circ}$ can be decomposed into $F_{i}(q,v) \in \Sigma_{Q_{i}}^{\circ}, \, i=1,...,n$, such that $F(q,v) = (F_1(q,v), ..., F_n(q,v))$. In view of the definition of the direct sum of Dirac structures and $L=\sum_{i=1}^{n}L_{i}$, we see that equation \eqref{eq:hoj2} implies 
\[
\left(\left(\dot{q}_i, \dot{p}_i \right) , \left(-\pder{L_{i}}{q_{i}} - F_i , v\right)\right) \in D_{ \Delta _{Q_{i}} }(q_i,p_i).
\]
However, this implies
\[
((q_{i},p_{i},\dot{q}_{i},\dot{p}_{i}), \mathbf{d}_{D} L_{i}(q_{i},v_{i})-\pi_{Q_{i}}^{\ast}F_{i}(q,v)) \in D_{ \Delta _{Q_{i}} }(q_{i},p_{i})
\]
for $i = 1,\dots,n$, together with the conditions
\begin{align*}
(\dot{q}_{1},...,\dot{q}_{n}) \in \Sigma_{Q}(q_{1},...,q_{n}) \;
\textrm{and}\;(F_{1}(q,v), \cdots, F_{n}(q,v)) \in \Sigma_{Q}^{\circ}(q_{1},...,q_{n}).
\end{align*}
We may reverse these steps to prove equivalence.
\end{proof}

As shown in the above, the interconnection of  $n$ distinct Lagrange-Dirac dynamical systems $(\mathbf{d}_{D} L_{i}, F_{i},D_{\Delta_{Q_{i}}})$ through $D_{\rm int}$ is equivalent with the Lagrange-Dirac dynamical system $(\mathbf{d}_{D} L,D_{\Delta_{Q}})$. 

\paragraph{Variational Structures for Interconnected Systems.}
Here, we consider the Lagrange-d'Alembert-Pontryagin variational structure for the interconnection of $n$ implicit Lagrangian subsystems.

\begin{definition}
The Lagrange-d'Alembert-Pontryagin principle for the interconnected mechanical systems is given for $i = 1 , \dots, n$ by
\begin{align}\label{LagdAlemPon1}
\delta \int_{t_1}^{t_2} L_{i}(q_{i}(t),v_{i}(t))&+\left< p_{i}(t), \dot{q}_{i}(t)-v_{i}(t) \right>  \,dt \\
&+ \int_{t_1}^{t_2} \left<F_{i}(q(t),v(t)), \delta q_{i}(t) \right>dt=0, \nonumber
\end{align}
for curves $(q_i(t),v_i(t),p_i(t)) \in TQ_i \oplus T^{\ast}Q_{i},\;t \in [t_1, t_2]$ with variations $\delta q_{i}(t) \in \Delta_{Q_{i}}(q_{i}(t))$ with fixed end points, arbitrary variations $\delta v_{i},\delta p_{i}$ and with $\dot{q}_{i}(t) \in \Delta_{Q_{i}}(q_{i}(t))$, and the condition
\begin{align}\label{CondInterconnection}
(\dot{q}_{1},...,\dot{q}_{n}) \in \Sigma_{Q}(q_{1},...,q_{n}) \;
\textrm{and}\;(F_{1}(q,v), \cdots, F_{n}(q,v)) \in \Sigma_{Q}^{\circ}(q_{1},...,q_{n}).
\end{align}
\end{definition}
\begin{proposition}
The interconnection of the Lagrange-d'Alembert-Pontryagin structures through $\Sigma_Q$ given in \eqref{LagdAlemPon1} and \eqref{CondInterconnection} for curves $(q_{i}(t),$ $v_{i}(t),p_{i}(t))$ in $TQ_{i} \oplus T^{\ast}Q_{i}, i = 1,\dots,n$ is equivalent to the Lagrange-d'Alembert-Pontryagin principle for the interconnected mechanical system is equivalent with the following one:
\begin{equation}\label{LagdAlemPon2}
\begin{split}
\delta \int_{t_1}^{t_2}
L(q(t),v(t)) + \left< p(t), \dot{q}(t)-v(t) \right>
\,dt=0,
\end{split}
\end{equation}
for a curve $(q(t),v(t),p(t))$ in $TQ \oplus T^{\ast}Q$ with variations $\delta{q}(t) \in \Delta_{Q}(q(t)) \subset T_{q(t)}Q$ with fixed endpoints, arbitrary unconstrained variations $\delta{v}(t)$ and $\delta{p}(t)$, and $\dot{q}(t) \in \Delta_{Q}(q(t)) \subset T_{q(t)}Q$. 
\end{proposition}
\begin{proof}
It follows from \eqref{LagdAlemPon1} that
\begin{equation}
\begin{split}\label{IntLagdAlemPonEq}
\dot{q}_{i}=v_{i} \in \Delta_{Q_{i}}(q_{i}), \quad  \dot{p}_{i}-\frac{\partial L_{i}}{\partial q_{i}}-F_{i} \in \Delta^{\circ}_{Q_{i}}(q_{i}),\quad p_{i}=\frac{\partial L_{i}}{\partial v_{i}},\quad i=1,...,n.
\end{split}
\end{equation}
Recall that the distribution $(\Delta_{Q_{1}} \times \cdots \times \Delta_{Q_{n}})(q_{1},\cdots,q_{n})=\Delta_{Q_{1}}(q_{1}) \times \cdots \times \Delta_{Q_{n}}(q_{n}) \subset TQ$ has the annihilator
$
(\Delta_{Q_{1}} \times \cdots \times \Delta_{Q_{n}})^{\circ}(q_{1},...,q_{n})=\Delta_{Q_{1}}^{\circ}(q_{1}) \times  \dots \times \Delta_{Q_{n}}^{\circ}(q_{n}),
$
and impose the additional constraints 
$$
(\dot{q}_{1},...,\dot{q}_{n}) \in \Sigma_{Q}(q_{1},...,q_{n}) \quad \textrm{and} \quad (F_{1}(q,v),...,F_{n}(q,v)) \in \Sigma_{Q}^{\circ}(q_{1},...,q_{n}), 
$$ 
one can develop the equations
\begin{align*}
(\dot{q}_{1},...,\dot{q}_{n}) = (v_{1},...,v_{n}) \in \Delta_{Q}(q_{1},...,q_{n}), \quad
\left( \dot{p}_{1}-\frac{\partial L_{1}}{\partial q_{1}},..., \dot{p}_{n}-\frac{\partial L_n}{\partial q_{n}} \right) \in \Delta_{Q}^{\circ}(q_{1},...,q_{n}),
\end{align*}
together with the Legendre transformation
$$
\left( p_{1}, ..., p_{2}\right) = \left(\frac{\partial L_{1}}{\partial v_{1}}, ..., \frac{\partial L_{2}}{\partial v_{2}} \right),
$$
where $\Delta_{Q}(q_{1},...,q_{n}) =(\Delta_{Q_{1}} \times \cdots \times \Delta_{Q_{n}})(q_{1},...,q_{n}) \cap \Sigma_{Q}(q_{1},...,q_{n}) \subset TQ$ is the final distribution and its annihilator is given by 
\[
\Delta^{\circ}_{Q}(q_{1},...,q_{n})=(\Delta_{Q_{1}} \times \cdots \times \Delta_{Q_{n}})^{\circ}(q_{1},...,q_{n})  +\Sigma_{Q}^{\circ}(q_{1},...,q_{n}).
\]
  Reflecting upon the last group of equations, one obtains the Lagrange-d'Alembert-Pontryagin equations \eqref{IntLagdAlemPonEq}, which can be also derived from the Lagrange-d'Alembert-Pontryagin principle in \eqref{LagdAlemPon2}.  The converse is proven by reversing the above arguments to prove the existence of the interaction forces $F_1,\dots,F_n$.
\end{proof}
It was already shown in \cite{YoMa2006b} that Lagrange-Dirac dynamical systems satisfy the Lagrange-d'Alembert-Pontryagin.  In Proposition \ref{prop:LDDS_splitting}, we illustrated how the equations for a Lagrange-Dirac dynamical system are coupled by an interaction Dirac structure of the form $D_{\mathrm{int}} = \Sigma_{\mathrm{int}} \oplus \Sigma^\circ_{\mathrm{int}}$ by introducing constraint forces.  The same constraint forces allow us to rewrite the Lagrange-d'Alembert-Pontryagin for an interconnected system as a set of the Lagrange-d'Alembert-Pontryagin principles for the separate primitive subsystems. 
\medskip

We can summarize these results in the following theorem:
  
\begin{theorem}
\label{thm:LDDS_splitting}
  Assume the same setup as Proposition \ref{prop:LDDS_splitting} and let $(q,v,p)(t),\; t \in [t_1,t_2]$ be a curve in $TQ \oplus T^{\ast}Q$. Then, the following statements are equivalent:
  \begin{enumerate}
  	\item[(i)] The curve $(q,v,p)$ satisfies
	\begin{align*}
	((q,p,\dot{q},\dot{p}), \mathbf{d}_{D} L(q,v)) \in D_{ \Delta _Q }(q,p).
	\end{align*}
	\item[(ii)] There exists some constraint force field $F_i: TQ \to T^{\ast}Q_i$ such that the curves $(q_{i},v_{i},p_{i})(t) \in TQ_i \oplus T^{\ast}Q_i$ satisfy
	\begin{align*}
((q_{i},p_{i},\dot{q}_{i},\dot{p}_{i}), \mathbf{d}_{D} L_{i}(q_{i},v_{i})-\pi_{Q_{i}}^{\ast}F_{i}(q,v)) \in D_{ \Delta _{Q_{i}} }(q_{i},p_{i}),
	\end{align*}
        for $i = 1,\dots n$, together with
        $
        (\dot{q}_{1},...,\dot{q}_{n}) \in \Sigma_{Q}(q_{1},...,q_{n})
        $
        and
        $$
        (F_{1}(q,v),..., F_{n}(q,v)) \in \Sigma_{Q}^{\circ}(q_{1},...,q_{n}).
        $$

	\item[(iii)] The curve $(q,v,p)(t)$ satisfies the Lagrange-d'Alembert-Pontryagin principle:
	\[
		\delta \int_{t_1}^{t_2}{ L(q,v) + \lb p , \dot{q} - v \rb dt} = 0
	\]
	with respect to chosen variations $\delta q(t) \in \Delta_{Q}(q(t))$ with fixed endpoints, $\delta v, \delta p$ arbitrary, and the constraint $\dot{q}(t) \in \Delta_{Q}(q(t))$.
	
	\item[(iv)] The curves $(q_{i},v_{i},p_{i})(t) \in TQ_i \oplus T^{\ast}Q_i$ satisfy the Lagrange-d'Alembert-Pontryagin principles:
	\[
		\delta \int_{t_1}^{t_2}{ L_i(q_i,v_i) + \lb p_i, \dot{q}_i - v_i \rb dt} + \int_{t_1}^{t_2}{ \lb F_i , \delta q \rb dt} = 0,
	\]
        for $i = 1,\dots,n$, together with
        $
        (\dot{q}_{1},...,\dot{q}_{n}) \in \Sigma_{Q}(q_{1},...,q_{n})
        $
        and
        $$
        (F_{1}(q,v),..., F_{n}(q,v)) \in \Sigma_{Q}^{\circ}(q_{1},...,q_{n}).
        $$
\end{enumerate}
\end{theorem}

\section{Examples}
\label{sec:examples}

  The unifying theme of interconnection is that we often find ourselves in a situation where we have a number of systems which we understand well (such as the components of a circuit or a rigid body),
  while the interconnected system is less understood.
  Therefore the concept of interconnection is useful because it allows us to use our previous knowledge of the subsystems to construct the interconnected system.
  These interconnections can be, geometrically speaking, quite sophisticated (e.g. interconnection by nonholonomic constraints).
  In this section, we provide some examples of interconnection of Lagrange-Dirac dynamical systems.
  We have chosen simple examples to illustrate the essential ideas of interconnection concretely.
  However, the method of tearing and interconnecting subsystems can extend to more complicated systems.
  
\medskip

\noindent
{\bf (I) A Mass-Spring Mechanical System.}\\
\noindent
Consider a mass-spring system as in Figure \ref{Mass-Spring}. Let $m_{i}$ and $k_{i}$ be the $i$-th mass and spring for $i=1,2,3$.

\begin{figure}[h]
\begin{center}
\includegraphics[scale=.5]{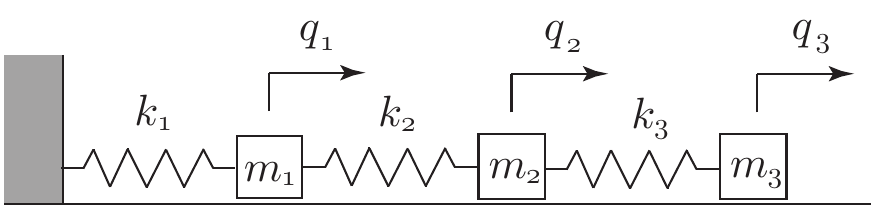}
\caption{A Mass-Spring System}
\label{Mass-Spring}
\end{center}
\end{figure}
\paragraph{Tearing and Interconnecting.}
Inspired by the concept of {\it tearing and interconnecting systems} developed by \cite{Kr1963}, the mass-spring mechanical system can be torn apart into two distinct subsystems called ``primitive systems''as in Figure \ref{TornApart}.  The procedure of tearing inevitably yields {\it interactive boundaries}, through which the energy flows between the primitive subsystem 1 and the primitive subsystem 2.  Upon tearing, the separate primitive systems obey the following condition at the interaction boundaries:
\begin{equation}\label{continuity}
f_{2}+\bar{f}_{2}=0, \qquad \dot{q}_{2}=\dot{\bar{q}}_{2}.
\end{equation}
In the above, $\dot{q}_2$ and $\dot{\bar{q}}_2$ are the associated velocities to the boundaries, while $f_{2}$ and $\bar{f}_{2}$ are the interaction forces.  We call equation \eqref{continuity} the {\it continuity condition}.  Without the continuity condition, there exists no energy interaction between the primitive subsystems.
\begin{figure}[h]
\begin{center}
\includegraphics[scale=.5]{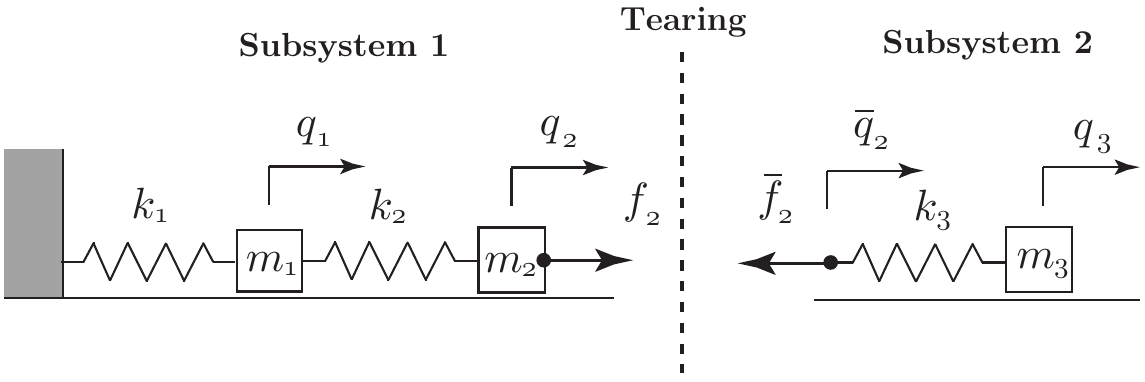}
\caption{Torn-apart Systems}
\label{TornApart}
\end{center}
\end{figure}
In other words, {\it the original mechanical system can be recovered by interconnecting the primitive subsystems with the continuity conditions}.
\medskip

Equation \eqref{continuity} implies that {\it power invariance} holds:
\[
\left<f_{2}, \dot{q}_{2} \right> +\left<\bar{f}_{2}, \dot{\bar{q}}_{2} \right> =0.
\]
Needless to say, the above equation may be understood by an interaction Dirac structure as shown later. 

\paragraph{Lagrangians for Primitive Systems.}
Let us consider how dynamics of the primitive systems can be formulated as forced Lagrange-Dirac dynamical systems. 
\medskip

The configuration space of the primitive system 1 may be given by $Q_{1}=\mathbb{R} \times \mathbb{R}$ with local coordinates $(q_{1},q_{2})$, while the configuration space of the primitive system 2 is $Q_{2}=\mathbb{R} \times \mathbb{R}$ with local coordinates $(\bar{q}_{2},q_{3})$. We can invoke the canonical Dirac structures $D_{TQ_{1}} \in \mathrm{Dir}(T^{\ast}Q_{1})$ and $D_{TQ_{2}} \in \mathrm{Dir}(T^{\ast}Q_{2})$ in this example.  For Subsystem 1, the Lagrangian $L_{1}:TQ_{1}\to \mathbb{R}$ is given by, for $(q_{1},q_{2},v_{1},v_{2}) \in TQ_{1}$,
\[
\begin{split}
L_{1}(q_{1},q_{2},v_{1},v_{2})=\frac{1}{2} m_{1}v_{1}^{2}+\frac{1}{2}m_{2}v_{2}^{2} - \frac{1}{2}k_{1}{q}_{1}^{2}-\frac{1}{2}k_{2}(q_{2}-q_{1})^{2},
\end{split}
\] 
while the Lagrangian $L_{2}:TQ_{2}\to \mathbb{R}$ for the primitive system 2 is given by, for $(\bar{q}_{2},q_{3},\bar{v}_{2},v_{3}) \in TQ_{2}$,
\[
\begin{split}
L_{2}(\bar{q}_{2},q_{3},\bar{v}_{2},v_{3})=\frac{1}{2} m_{3}v_{3}^{2}- \frac{1}{2}k_{3}({q}_{3}-\bar{v}_{2})^{2}.
\end{split}
\] 

When viewing each system separately, the constraint force acts as an external force on each primitive system. Again, this is because {\it tearing always yields constraint forces at the boundaries associated with the disconnected primitive systems}, as shown in Figure \ref{TornApart}
\medskip

\noindent
{\bf Primitive System 1.}
  Given an interaction force $F_{1}: TQ \to T^*Q_1$, we can formulate equations of motion for the Lagrange-Dirac dynamical system $(\mathbf{d}_{D}L_{1}, F_{1},D_{TQ_{1}})$ by
\begin{equation}\label{prim1}
\begin{split}
\dot{q}_{1}=v_{1},\quad  \dot{q}_{2}=v_{2},\quad \dot{p}_{1}=-k_{1}q_{1}-k_{2}(q_{1}-q_{2}),\quad
\dot{p}_{2}=k_{2}(q_{1}-q_{2})+f_{2}(q,v),
\end{split}
\end{equation}
together with $p_{1}=m_{1}v_{1}$ and $p_{2}=m_{2}v_{2}$ and 
\[
	F_{1}(q,v) = (q_{1},q_{2},0,f_{2}(q,v)).
\]
where $(q,v) = (q_1,q_2,\bar{q}_2,q_3,v_1,v_2, \bar{v}_2, v_3 ) \in TQ$.  This implicit Lagrange-d'Alembert equation is well defined when we are given $(q_2(t),v_2(t)) \in TQ_2$.
\medskip

\noindent
{\bf Primitive System 2.}
Similarly, by introducing an interaction force, $F_2 : TQ \to T^*Q_2$, on the port variable $\bar{q}_2$ we can also formulate equations of motion for the Lagrange-Dirac dynamical system $(\mathbf{d}_{D}L_{2}, F_{2},D_{TQ_{2}})$ by
\begin{equation}\label{prim2}
\begin{split}
\dot{\bar{q}}_{2}=\bar{v}_{2},\quad
\dot{q}_{3}=v_{3},\quad
\dot{\bar{p}}_{2}=k_{3}(q_{3}-\bar{q}_{2})+\bar{f}_{2},\quad
\dot{p}_{3}=-k_{3}(q_{3}-\bar{q}_{2}),
\end{split}
\end{equation}
together with 
\[
	F_2(q,v) = (\bar{q}_2,q_3,\bar{f}_2(q,v),0),
\]
and the primary constraints $\bar{p}_{2}=0$ and $p_{3}=m_{3}v_{3}$ as well as the consistency condition, $\dot{\bar{p}}_{2}=0$, where $(q,v) = (q_1,q_2,\bar{q}_2,q_3,v_1,v_2, \bar{v}_2, v_3) \in TQ$.  Again, this implicit Lagrange-d'Alembert equation is well defined when we are given $(q_1(t) , v_1(t) ) \in TQ_1$.
\medskip

In the next paragraph, we will interconnect these separate primitive systems to reconstruct the original mass-spring system through an interaction  Dirac structure.

\paragraph{Interconnection of Separate Dirac Structures.} 
Let $Q=Q_{1} \times Q_{2} = \mathbb{R} \times \mathbb{R} \times \mathbb{R} \times \mathbb{R}$ be an {\bfi extended configuration space} with local coordinates $q=(q_{1},q_{2},\bar{q}_{2},q_{3})$. Recall that the direct sum of the induced Dirac structures is given by $D_{TQ_{1}} \oplus D_{TQ_{2}}$ on $T^{\ast}Q$. The constraint distribution due to the interconnection is given by 
\[
\Sigma_{Q}(x)= \{ v \in T_{x}Q \mid \left<\omega_{Q}(x), v \right>=0 \},
\]
where $\omega_{Q}=d{q}_{2}-d\bar{q}_{2}$ is a one-form on $Q$. On the other hand, the annihilator $\Sigma_{Q}^{\circ} \subset T^{\ast}Q$ is defined by
\[
\begin{split}
\Sigma_{Q}^{\circ}(q)=\{ f=(f_{1},f_{2},\bar{f}_{2},f_{3}) \in T^{\ast}_{x}Q \mid 
\; \left<f, v\right>=0 \;\mbox{and}\; v \in \Sigma_{Q}(x) \}.
\end{split}
\]
It follows from this codistribution that
$f_{2}=-\bar{f}_{2}$, $f_{1}=0$ and $f_{3}=0$.
Hence, we obtain the conditions for the interconnection given by \eqref{continuity}; namely, $f_{2}+\bar{f}_{2}=0$ and $v_{2}=\bar{v}_{2}$. 
Let  $\Sigma_{\mathrm{int}}=(T\pi_{Q})^{-1}(\Sigma_{Q}) \subset TT^{\ast}Q$ and let $D_{\mathrm{int}}$ be defined as in \eqref{IntDirac}. Finally we derive the interconnected Dirac structure $D_{\Delta_{Q}}$  on $T^{\ast}Q$ given by
\[
 D_{\Delta_{Q}}=(D_{TQ_{1}} \oplus D_{TQ_{2}})  \boxtimes D_{\mathrm{int}}.
\]

\paragraph{Interconnection of Primitive Systems.} 
Now, let us see how decomposed primitive systems can be interconnected to recover the original mechanical system.  Define the Lagrangian $L:TQ\to \mathbb{R}$ for the interconnected system by 
$
L=L_{1}+L_{2}.
$
Let $\Delta_{Q}=(TQ_{1} \times TQ_{2}) \cap \Sigma_{\mathrm{int}}$. Then, equations of motion for the interconnected Lagrange-Dirac dynamical system may be given by a set of equations \eqref{prim1}, \eqref{prim2} and \eqref{continuity}, which are finally given in matrix by 
\begin{align*}
&\left(
\begin{array}{cccc|cccc}
0 & 0 & 0 & 0 & -1 & 0 & 0 & 0 \\
0 & 0 & 0 & 0 & 0 & -1 & 0 & 0 \\
0 & 0 & 0 & 0 & 0 & 0 & -1 & 0 \\
0 & 0 & 0 & 0 & 0 & 0 & 0 & -1 \\
\hline \vspace{-3mm} \\
1 & 0 & 0 & 0 & 0 & 0 & 0 & 0 \\
0 & 1 & 0 & 0 & 0 & 0 & 0 & 0 \\
0 & 0 & 1 & 0 & 0 & 0 & 0 & 0 \\
0 & 0 & 0 & 1 & 0 & 0 & 0 & 0 \\
\end{array}
\right)
\left(
\begin{array}{c}
\dot{q}_{1} \\
\dot{q}_{2} \\
\dot{\bar{q}}_{2} \\
\dot{q}_{3} \vspace{1mm} \\
\hline \vspace{-3mm} \\
\dot{p}_{1} \\
\dot{p}_{2} \\
\dot{\bar{p}}_{2} \\
\dot{p}_{3} \\
\end{array}
\right) 
=
\left(
\begin{array}{c}
k_{1}x_{1}-k_{2}(q_{2}- q_{1}) \vspace{1mm} \\
k_{2}x_{2}\\
-k_{3}(q_{3}-\bar{q}_{2})\vspace{1mm} \\
k_{3}(q_{3}-\bar{q}_{2})\vspace{1mm} \\
\hline \vspace{-3mm} \\
v_1 \\
v_2 \\
\bar{v}_{2} \\
v_{3} \\
\end{array}
\right)+
\left(
\begin{array}{cc}
0 \\
  -1 \\
  1 \\
  0 \\
\hline \vspace{-3mm} \\
 0 \\
  0 \\
  0 \\
  0 \\  
\end{array}
\right)
f_{2},
\end{align*}
together with the Legendre transformation $p_{1}=m_{1}v_{1}, p_{2}=m_{2}v_{2}, \bar{p}_{2}=0, p_{3}=m_{3}v_{3}$,
the interconnection constraint $v_{2}=\bar{v}_{2}$,
as well as the consistency condition
$\dot{\bar{p}}_{2}=0$.
\medskip

\noindent
{\bf (II) Electric Circuits}\\
\noindent
Consider the electric circuit depicted in Figure \ref{vrlc_circuit_pic}, where $R$ denotes a resistor, $L$ an inductor, and $C$ a capacitor. 

\begin{figure}[h] 
   \centering
\includegraphics[scale=.8]{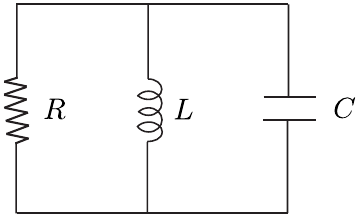} 
   \caption{R-L-C Circuit}
   \label{vrlc_circuit_pic}
\end{figure}

As in Figure \ref{torn_circuit_pic}, we decompose the circuit into two disconnected primitive systems.  Let $S_{1}$ and $S_{2}$ denote external ports resulting from the tear. In order to reconstruct the original circuit  in Figure \ref{vrlc_circuit_pic}, the external ports may be connected by equating currents across each.  
\begin{figure}[h] 
   \centering
\includegraphics[scale=.8]{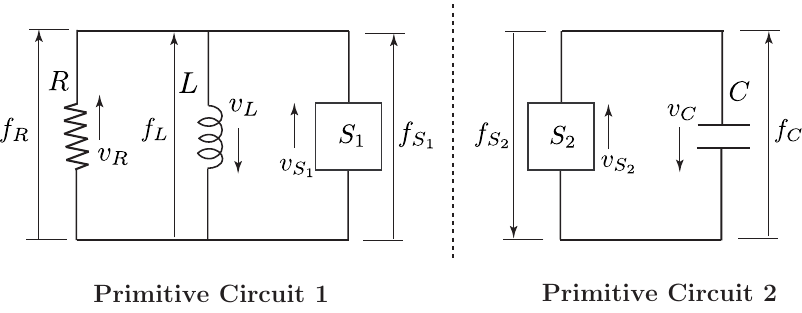} 
   \caption{Primitive Circuits}
   \label{torn_circuit_pic}
\end{figure}

\paragraph{Primitive System 1.}
The configuration space for the primitive system 1 is denoted by $Q_{1}=\mathbb{R}^3$ with local coordinates
$q_1 =(q_{R}, q_{L}, q_{S_1})$, where $q_R, q_L$ and $q_{S_{1}}$ are the charges associated to the resistor $R$, inductor $L$ and port $S_{1}$.  Kirchhoff's circuit law is enforced  by applying a constraint distribution $\Delta_{Q_{1}} \subset TQ_{1}$, which is given by, for each $q_1 =(q_{R}, q_{L}, q_{S_1}) \in Q_{1}$,
\[
	\Delta_{Q_{1}}(q_{1}) = \{ v_{1}=(v_{R}, v_{L}, v_{S_1}) \in T_{q_{1}}Q_1 \mid  v_R - v_L
	 - v_{S_1} = 0 \},
\]
where $v_{1}=(v_{R}, v_{L}, v_{S_1})$ denotes the current vector at each $q_{1}$,
while the KVL constraint is given by its annihilator $\Delta_{Q_{1}}^\circ$, which is given by, for each $q_1 =(q_{R}, q_{L}, q_{S_1}) \in Q_{1}$,
\[
	\Delta_{Q_{1}}^{\circ}(q_{1}) = \{ f_{1}=(f_{R}, f_{L}, f_{S_1}) \in T_{q_{1}}^{\ast}Q_1 \mid  f_R = f_L= f_{S_1}  \}.
\]

Then, we can naturally define the induced Dirac structure $D_{\Delta_{Q_{1}}}$ on $T^{\ast}Q_{1}$ from $\Delta_{Q_{1}}$ as before.
\medskip

For the primitive circuit 1, the Lagrangian $\mathcal{L}_{1}$ on $TQ_{1}$ is given by
\[
	\mathcal{L}_1(q_1,v_1) = \frac{1}{2} L_{1}v_L^2,
\]
which is degenerate. The voltage associated to the resistor $R$ may be given by
\[
	f_{R}(q_R,v_R) = (q_R, -R v_R),
\]
while the voltage associated to the port $S_1$ is denoted by $f_{S_{1}}(q_{S_{1}},v_{S_{1}})dq_{S_{1}}$.  Since the interaction voltage field $F_{1}: TQ \to T^{\ast}Q_{1}$ for the primitive circuit 1 is given by 
\[
F_{1}(q , v  )=(q_{R}, q_{L}, q_{S_1},f_{R}(q_R,v_R),0,f_{S_{1}}(q,v)), 
\]
for $(q,v) = (q_{R}, q_{L}, q_{S_1}, q_{S_2}, q_{C},v_{R}, v_{L}, v_{S_1}, v_{S_2}, v_{C}) \in TQ$.
We can set up equations of motion for $(\mathbf{d}_{D}L_{1}, F_{1},D_{\Delta_{Q_1}})$ as
$$
((q_{1},p_{1},\dot{q}_{1},\dot{p}_{1}), \mathbf{d}_{D} L_{1}(q_{1},v_{1})-\pi_{Q_{1}}^{\ast}F_{1}(q,v)) \in D_{ \Delta _{Q_{1}} }(q_{1},p_{1}),
$$
and expressed more explicitly as
\begin{equation}\label{primCir1}
\begin{split}
\dot{q}_{R}=v_{R},\quad  \dot{q}_{L}=v_{L},\quad \dot{q}_{S_{1}}=v_{S_{1}},\quad -f_{R}=\lambda_{1}, \quad \dot{p}_{L}=-\lambda_{1},\quad
f_{S_{1}}=\lambda_{1},
\end{split}
\end{equation}
together with $p_{L}=Lv_{L}$, $p_{R}=0$, $p_{S_{1}}=0$, $\dot{p}_{R}=0$ and $\dot{p}_{S_{1}}=0$.  These equations of motion are well defined when we are given $(q_2(t), v_2(t) ) \in TQ_2$.

\paragraph{Primitive System 2.}
The configuration space for the primitive system 2 is $Q_{2}= \mathbb{R}^2$ with local coordinates
$q_{2}=(q_{S_2},q_C)$, where $q_{S_2}$ is the charge through the port $S_2$ and $q_C$ is the charge stored in the capacitor.  
The KCL space is given by, for each $q_{2}=(q_{S_2},q_C) \in Q_{2}$, 
\[
	\Delta_{Q_{2}}(q_{2}) = \{ v_2=(v_{S_2},v_C) \in T_{q_{2}}Q_2 \mid v_C - v_{S_2} = 0 \}
\]
and hence the KVL space is given by the annihilator $\Delta_2^\circ(q_{2})$ as
\[
	\Delta^{\circ}_{Q_{2}}(q_{2}) = \{ f_2=(f_{S_2},f_C) \in T_{q_{2}}^{\ast}Q_2 \mid f_C = f_{S_2} \}.
\]

This gives us the Dirac structure $D_{2}$ on $T^{\ast}Q_{2}$. Set the Lagrangian $\mathcal{L}_2:TQ_{2} \to \mathbb{R}$ for Circuit 2 to be
\[
	\mathcal{L}_2 = \frac{1}{2C} q_{C}^2.
\]
Given an interaction voltage field for the primitive system 2 as $F_{2}(q,v)=(q_{S_2},q_C,f_{S_{2}}(q,v),0)$, we can formulate the equations of motion of $(\mathbf{d}_{D}L_{2}, F_{2},D_{\Delta_{Q_2}})$ as
$$
((q_{2},p_{2},\dot{q}_{2},\dot{p}_{2}), \mathbf{d}_{D} L_{2}(q_{2},v_{2})-\pi_{Q_{2}}^{\ast}F_{2}(q,v)) \in D_{ \Delta _{Q_{2}} }(q_{2},p_{2}),
$$
which are given by
\begin{equation}\label{primCir2}
\begin{split}
\dot{q}_{S_{2}}=v_{S_{2}},\quad  \dot{q}_{C}=v_{C},\quad -\frac{q_{C}}{C}= f_{S_2}(q,v) ,
\end{split}
\end{equation}
together with $p_{S_{2}}=0$, $p_{C}=0$, $\dot{p}_{S_{2}}=0$ and $\dot{p}_{C}=0$.  These equations of motion are well defined when we are given $(q_1(t), v_1(t) ) \in TQ_1$.

\paragraph{The Interaction Dirac Structure.}
Set $Q = Q_1 \times Q_2$ and given
\[
	\Sigma_{Q} = \{ (v_R,v_L,v_{S_{1}},v_{S_{2}}, v_C) \in TQ \mid v_{S_1}=v_{S_2}  \},
\]
and with the annihilator 
$$
\Sigma_{Q}^{\circ} = \{ (0,0,f_{S_{1}},f_{S_{2}},0) \in T^{\ast}Q \mid f_{S_1}+f_{S_2}=0  \}.
$$
Setting $D_{Q} = \Sigma_{Q}\oplus \Sigma_{Q}^{\circ}$, we can define the interaction Dirac structure, $D_{\mathrm{int}} = \pi^{\ast}_{Q}D_{Q}$,
which is denoted, locally, by
\begin{align*}
	D_{\mathrm{int}} (q,p)& = \{ (\dot{q},\dot{p}), (\alpha ,w))  \in T_{(q,p)}{T^*Q} \times T^{\ast}_{(q,p)}T^{\ast}Q \mid \\
		 &\hspace{5cm} \dot{q}_{S_1} = \dot{q}_{S_2}, \; w_1 = 0,\; w_2 = 0, \;
		 \alpha_{S_{1}}+\alpha_{S_{2}}=0 \}.
\end{align*}
where $q = (q_R, q_L, q_{S_1},q_{S_2},q_C) , p = (p_R, p_L, p_{S_1},p_{S_2},p_C),\alpha = (\alpha_R, \alpha_L, \alpha_{S_1},\alpha_{S_2},\alpha_C)$, and $w = (w_R, w_L, w_{S_1},w_{S_2},w_C)$.

In this way, the velocity $v=(v_R, v_L, v_{S_{1}},v_{S_{2}}, v_C)$ and force $f_{S}=(0,0,f_{S_{1}},f_{S_{2}},0)$ at the boundaries hold the constraint, $(v,f_S) \in D_{Q}$.
Thus, the equations of motion for the interconnected Lagrange-Dirac dynamical system are given by a set of equations \eqref{primCir1} and \eqref{primCir2} together with 
$v_{S_{1}}=v_{S_{2}}$ and $f_{S_{1}}+f_{S_{2}}=0$.

\medskip

\noindent
{\bf (III) A Ball Rolling on Rotating Tables}\\
\noindent
Consider the mechanical system depicted in Figure \ref{fig:gear_system}, where there are two rotating tables and a ball is rolling on one of the tables without slipping. We assume that the gears are ideally linked by a non-slip constraint without any loss of energy and hence the mechanical system is conservative.  Let $I_{1}$ and $I_{2}$ be moments of inertia for the tables. We will now decompose the system into two primitive systems; (1) a rolling ball on a rotating (large) table and (2) a rotating (large) table. By tearing the mechanical system into the primitive systems, 
there may appear (constraint) torques $\tau_{s_{1}}$ and $\tau_{s_{2}}$ associated with the angular velocities $\dot{s}_{1}$ and $\dot{s}_{2}$.
Later, we will show how the constraint torques as well as the angular velocities at the contact point of the rotating tables can be incorporated into an interaction Dirac structure.
\begin{figure}[h]
\begin{center}
\includegraphics[scale=1]{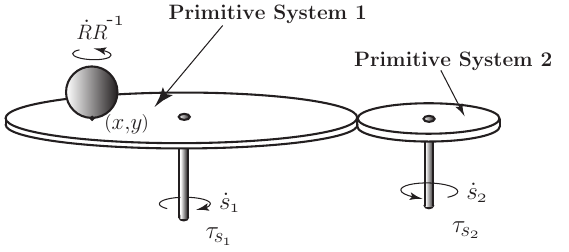}
\caption{A Rolling Ball on Rotating Tables without Slipping }
\label{fig:gear_system}
\end{center}
\end{figure}

\paragraph{Primitive System 1.}
The configuration space for the primitive system 1, namely, a ball of unit radius rolling on a rotating (large) table is given by $Q_{1}=\mathbb{R}^{2} \times SO(3) \times S^{1}$ (see, for instance, \cite{BKMM1996,LewMur1995}), where we denote a point in $Q_{1}$ by $q_{1}=(x,y,R, s_{1})$. Here $(x,y) \in \mathbb{R}^{2}$ denotes the position of the contact point 
of the ball with respect to the center of rotation of the  table, $R$ is the rotational matrix in $SO(3)$ and $s_{1}$ denotes the rotation angle about the shaft.  The ball is assumed to be a sphere with uniform mass density. So, let $m$ and $I$ be the mass and the moment of inertia of the ball.  Let $I_{s_{1}}$ be the moment of inertia of the large table about the vertical axis. Then, the Lagrangian of the ball and rotating table $L_{1}: TQ_{1}\to \mathbb{R}$ is given by, for $(q_{1},v_{1})=(x,y,R, s_{1},v_{x},v_{y},v_{R},v_{s_{1}}) \in TQ_{1}$,
$$
L_{1}(q_{1},v_{1})=\frac{1}{2}m(v_x^{2} + v_y^2) +\frac{1}{2}I\mathrm{tr}\left(v_{R}R^{-1} \cdot v_{R}R^{-1}\right)+\frac{1}{2}I_{s_{1}}||v_{s_{1}}||^{2}.
$$ 
The ball is rolling on the table without slipping and hence we have the nonholonomic constraints as follows (see \cite[page 800]{LewMur1995}):
\begin{align*}
	\Delta_{Q_{1}}(q_{1}) = \{  (v_x, v_y ,v_{R},v_{s_{1}}) \in T_{q_{1}}Q_{1}  \mid \;\;& v_{x}- \mathbf{i} \cdot v_{R} R^{-1} \cdot \mathbf{k} = - v_{s_{1}} y,\;\;  v_{y}+ \mathbf{k} \cdot v_{R} R^{-1} \cdot \mathbf{j} =  v_{s_{1}} x \}.
\end{align*}
Then, we can define a Dirac structure $D_{1}$ on $T^{\ast}Q_{1}$ induced from the distribution $\Delta_{Q_{1}}$ as, for $(q_{1},p_{1})=(x,y,R, s_{1},p_{x},p_{y},p_{R},p_{s_{1}}) \in T^{\ast}Q_{1}$,
\begin{align*}
	D_{1}(q_{1},p_{1}) &= \{ \left((\delta q_{1}, \delta p_{1}), (\alpha_1, w_1)\right) \in T_{(q_{1}, p_{1})}(T^{\ast}Q_1) \times T^{\ast}_{(q_{1}, p_{1})}(T^{\ast}Q_1) \mid \\
		&\hspace{1cm} \delta{q}_{1}=w_{1} \in \Delta_{Q_{1}}(q_{1}),\;\;\alpha_{1}+\delta{p}_{1} \in \Delta^{\circ}_{Q_{1}}(q_{1}) \}.
\end{align*}

By decomposition, the torque $\tau_{s_{1}}$ about the shaft may be regarded as an external force $F_{1}(q,v)=(0,0,\tau_{s_{1}} (q,v) )$ for the primitive system 1. Then, the equations of motion for $(D_{1}, \mathbf{d}_{D}L_{1}, F_{1})$ may be obtained from
$$
((q_{1},p_{1},\dot{q}_{1},\dot{p}_{1}), \mathbf{d}_{D} L_{1}(q_{1},v_{1})-\pi_{Q_{1}}^{\ast}F_{1}(q,v)) \in D_{1}(q_{1},p_{1}).
$$
 
 \paragraph{Primitive System 2.}
The configuration manifold for the primitive system 2 is a circle, $Q_2 = S^1$ and we set a point $q_{2}=s_{2} \in Q_{2}$. By left trivialization we interpret $TQ_{2}=TS^1$ as $S^1 \times \mathbb{R}$, and the Lagrangian  for the primitive system 2 is given by the rotational kinetic energy as, for $(q_{2},v_{2})=(s_{2},v_{s_{2}}) \in TQ_{2}$,
\[
	L_2(q_{2},v_{2}) = \frac{I_2}{2} v_{s_{2}}^2.
\]
Again, we have the canonical Dirac structure $D_{2}$ on $T^{\ast}Q_{2}$ as, for each $(q_{2},p_{2})=(s_{2},p_{s_{2}})$,
\[
	D_2 (q_{2},p_{2})= \{ ( \delta{s}_2, \delta{p}_{s_2} , \alpha_{s_2}, w_{s_2}) \mid \delta{s}_2 = w_{s_2} , \delta{p}_{s_2} + \alpha_{s_2} = 0 \}.
\]
Setting the torque $\tau_{s_{2}}$ about the shaft as an external force $F_{2}(q,v)=\tau_{s_{1}}(q,v)$ for the primitive system 1. Then, the equations of motion for $(D_{2}, \mathbf{d}_{D}L_{2}, F_{2})$ may be obtained from
$$
((q_{2},p_{2},\dot{q}_{2},\dot{p}_{2}), \mathbf{d}_{D} L_{2}(q_{2},v_{2})-\pi_{Q_{2}}^{\ast}F_{2}(q,v)) \in D_{2}(q_{2},p_{2}).
$$

\paragraph{Interaction Dirac Structure.}
  Let $Q = Q_1 \times Q_2$ and $q =(q_{1},q_{2})=(x,y,R,s_{1},s_{2}) \in Q$.  In order to interconnect the two primitive systems, we need to impose the constraints due to the non slip conditions.   The interconnection constraint between the primitive system 1 and primitive system 2 is given by, for each $v=(v_{1},v_{2})=(v_{x},v_{y},v_{R},v_{s_{1}},v_{s_{2}}) \in T_{q}Q$,
\[
	\Sigma_{Q}(q) = \{ (v_{x},v_{y},v_{R},v_{s_{1}},v_{s_{2}}) \in T_{q}Q \mid v_{{s}_1} + v_{{s}_2} = 0 \}
\]
and with its annihilator
\[
	\Sigma_{Q} ^\circ(q) = \operatorname{span}(\omega_1)
\]
where $\omega_1 = ds_1 - ds_2$. Setting $D_{Q}=\Sigma_{Q} \oplus \Sigma_{Q}^{\circ} \subset TQ \oplus T^{\ast}Q$, we can define the interaction Dirac structure as $D_{\rm int}=\pi^{\ast}_{Q}D_{Q}$.

Upon interconnecting the two primitive systems, one needs to impose the constraint on $(v,F)=(v_{r},v_{R},v_{s_{1}},v_{s_{2}}, 0,0,\tau_{s_{1}},\tau_{s_{2}}) \in TQ \oplus T^{\ast}Q$ given by $(v,F) \in D_{Q}(q)$.
This constraint ensures that the gears rotate (without slipping) at the same speed in opposite directions and the constraint torques are in equilibrium.

\paragraph{The Interconnected Lagrange-Dirac System.}
The Dirac structure for the interconnected system is given by
\[
	D_{\Delta_{Q}} = \left( D_1 \oplus D_2  \right) \boxtimes D_{\mathrm{int}}.
\]
Note that $D_{\Delta_{Q}}$ is defined by the canonical two-form on $T^{\ast}Q$ and the distribution 
\[
	\Delta_{Q} = (TQ_1 \oplus TQ_2) \cap \Sigma_{\mathrm{int}}.
\]
Additionally, the annihilator is given by $\Delta_{Q}^\circ =  \Sigma_{\mathrm{int}}^\circ$.
Setting $L= L_1 + L_2$, the dynamics of the interconnected Lagrange-Dirac dynamical system $(\extd L_{D},D_{\Delta_{Q}})$ may be given by,
\[
((q,p,\dot{q},\dot{p}),\mathbf{d}_{D}L(q,v)) \in D_{\Delta_{Q}}(q,p),
\]
for each $(q,v, p) \in TQ \oplus T^{\ast}Q$ where $p=\partial{L}/\partial{v}$.
\section{Conclusions}

  Tearing and interconnecting physical systems plays an essential role in modular modeling.
  In this paper we have shown how these concepts manifest themselves in the context of interconnection of Dirac structures and Lagrange-Dirac dynamical systems.
  In particular, it was shown how a Lagrange-Dirac dynamical system can be decomposed into primitive subsystems and how the primitive subsystems can be interconnected to recover the original Lagrange-Dirac dynamical system through an interaction Dirac structure.
  To do this, we first introduced the notion of {\it interconnection of Dirac structures} by employing the tensor product of Dirac structures $\boxtimes$.
  This process can be repeated $n$-fold due to the associativity of $\boxtimes$ (assuming the clean-intersection condition holds).
  This enables us to understand large heterogenous systems by decomposing them and keeping track of the relevant interaction Dirac structures.
  We also clarified how the variational principle for an interconnected system can be decomposed into variational structures on separate primitive subsystems which are coupled through boundary constraints on the velocities and forces.
  Lastly, we demonstrated our theory with the examples of a mass-spring system, an electric circuit, and a noholonomic mechanical system.
  The result of this study verifies a geometrically intrinsic framework for analyzing large heterogenous systems through tearing and interconnection.
\medskip

  We hope that the framework provided here can be explored further.  We are specifically interested in the following areas for future work:

\begin{itemize}
 \item {\bf \it The use of more general interaction Dirac structures:} We can consider presymplectic structures, such as those associated with gyrators, motors, magnetic couplings and so on (in this paper, we mostly studied interaction Dirac structures of the form $\Sigma_{\mathrm{int}} \oplus \Sigma^\circ_{\mathrm{int}}$). For some examples of these more general interconnections see \cite{WyCh1977, Yo1995}.

\item {\bf \it Reduction and symmetry for interconnected Lagrange-Dirac systems:}  The reduction of Lagrange-Dirac dynamical systems has been studied for Lie groups and cotangent bundles (\cite{YoMa2007b}, \cite{YoMa2009}).
Interpreting the curvature tensor of a principal connection as an interaction Dirac structure we may arrive at some interesting interpretations of magnetic couplings (for details on the curvature tensor see \cite{CeMaRa2001}). 

\item {\bf \it Interconnection of multi-Dirac structures and Lagrange-Dirac field systems:} In conjunction with classical field theories or infinite dimensional dynamical systems, the notion of multi-Dirac structures have been developed by \cite{VaYoLe2012}, which may be useful for the analysis of fluids, continuums as well as electromagnetic fields. The present work of the interconnection of Dirac structures and the associated Lagrange-Dirac systems may be extended to the case of classical fields or infinite dimensional dynamical systems.

\item {\bf \it Applications to complicated systems:}  For example, we could consider guiding central motion problems, multibody systems, fluid-structure interactions, passivity controlled interconnected systems, etc. (for examples of these systems see \cite{Little1983, Featherstone1987,JaVa2013,Yo1995,Van96} and \cite{OrtVanMasEsc02}).

  
  \item {\bf \it Discrete versions of interconnection and $\boxtimes$:}  By discretizing the Hamilton-Pontryagin principle one arrives at a discrete mechanical version of Dirac structures (see \cite{BoRaMa2009} and \cite{LeOh2010}).   A discrete version of $\boxtimes$ could allow for notions of interconnection of variational integrators. 
 \end{itemize}
 
\bibliographystyle{plainnat}
\bibliography{JaYo}
 
 \end{document}